\RequirePackage{fix-cm}
\documentclass[smallcondensed]{svjour3}
\smartqed
\usepackage{graphicx}

\usepackage[T1]{fontenc}
\usepackage[utf8]{inputenc}
\usepackage{float}

\usepackage{amsmath}
\usepackage{amssymb}
\usepackage{stmaryrd}
\usepackage{amsfonts}
\usepackage{dsfont}
\usepackage{subcaption}
\usepackage[english]{babel}

\usepackage{todonotes}

\newcommand{\dr}{\mathrm{d}}

\newcommand{\dd}[2]{\frac{\dr#1}{\dr#2}}
\newcommand{\pdd}[2]{\frac{\partial #1}{\partial #2}}

\newcommand{\bsprod}[2]{\left\langle #1,#2\right\rangle }
\newcommand{\evaluated}[2]{\left.#1\right|_{#2}}
\newcommand{\T}{\mathcal{T}}
\newcommand{\F}{\mathcal{F}}

\newcommand{\en}[1]{{\left\vert\kern-0.25ex\left\vert\kern-0.25ex\left\vert #1 
    \right\vert\kern-0.25ex\right\vert\kern-0.25ex\right\vert}}

\newcommand{\extension}{\mathds{E}}

\newcommand{\R}{\mathbb{R}}

\newcommand{\e}{\mathrm{e}}

\newcommand{\idx}[2]{{#1}_{#2}}
\newcommand{\pair}[1]{\{ \idx{#1}{1} , \idx{#1}{2} \}}
\newcommand{\jump}[1]{\llbracket #1 \rrbracket}

\begin{document}
\date{}
\title{High Order Cut Finite Elements for the Elastic Wave Equation
\thanks{This research was supported by the Swedish Research Council (Grant No. 2014-6088).}
}
\author{Simon Sticko \and Gustav Ludvigsson \and Gunilla Kreiss}

\institute{
S. Sticko, ORCiD: 0000-0002-4694-4731, \email{simon.sticko@it.uu.se} 
\and
G. Ludvigsson, \email{gustav.ludvigsson@it.uu.se}
\and
G. Kreiss, \email{gunilla.kreiss@it.uu.se}
\at Division of Scientific Computing, Department of Information Technology, Uppsala University, Sweden
}

\maketitle

\begin{abstract}
A high order cut finite element method is formulated for solving the elastic wave equation.
Both a single domain problem and an interface problem are treated.
The boundary or interface are allowed to cut through the background mesh.
To avoid problems with small cuts, stabilizing terms are added to the bilinear forms corresponding to the mass and stiffness matrix.
The stabilizing terms penalize jumps in normal derivatives over the faces of the elements cut by the boundary/interface.
This ensures a stable discretization independently of how the boundary/interface cuts the mesh.
Nitsche's method is used to enforce boundary and interface conditions, resulting in symmetric bilinear forms.
As a result of the symmetry, an energy estimate can be made and optimal order a priori error estimates are derived for the single domain problem.
Finally, numerical experiments in two dimensions are presented that verify the order of accuracy and stability with respect to small cuts.

\end{abstract}

\section{Introduction\label{sec:introduction}}
The elastic wave equation is important in several applications.
For example, materials in the earth's crust can be modeled as linear and elastic, and earthquakes give rise to seismic waves that propagate through the crust.
Other examples include non-destructive testing and propagation of waves in beams and other solid structures. 
High order accurate methods are especially attractive when solving the elastic wave equation.
The reason is that high order methods, in general, have lower work per dispersion error \cite{kreiss_comparison_1972}.
Seismic waves typically propagate over large distances and are therefore prone to dispersion error.
Also, elastic waves often propagate in media with complicated geometries. 
So, it is of interest to have numerical methods with high order of accuracy that can handle complicated geometries. 
Examples of such methods are discontinuous Galerkin (dG) methods \cite{de_basabe_interior_dg,riviere_discontinuous_2003}, and summation by parts (SBP) based finite difference methods \cite{duru2014stable,appelo2009stable}.
The dG methods usually handle the complicated geometries by using an unstructured grid that conforms to the boundary, meanwhile the SBP-based finite difference methods use a curvilinear grid to handle the complicated geometries.
While both of these methods work very well, it may at times be hard to find a good curvilinear mapping and it can be cumbersome to generate a good conforming grid.

In the present paper, we are interested in solving the elastic wave equation using the cut finite element method (Cut-FEM) with high order elements. 
Cut-FEM is an immersed method where boundaries and interfaces do not need to be aligned with the computational mesh.
For details on Cut-FEM see for example the review paper \cite{CutFEM2014}. 
Cut-FEM with high order elements has been studied in for example \cite{hansbo2017cut,johansson2015high,sticko2016higher}.
When using high order elements in Cut-FEM a few difficulties emerge.
One problem is generating a high order quadrature on the elements that are cut by the boundary or interface.
To accomplish this we use an algorithm by Saye \cite{Saye2015}.
Further, in the same way as for standard non-cut finite elements, the time step restriction becomes more severe when the element order is increased, which makes the time stepping more expensive.
Finally, stabilization terms are added (introduced in \cite{burman_ghost_2010,Burman2012,massing_stokes_2014}) in order to make the eigenvalues of the matrices bounded from above and below, independently of how the boundary/interface cuts the mesh.
Unfortunately, this stabilization makes the condition number of the mass matrix increase fast when the element order increases. 
This can make time stepping the discrete system more expensive since we need to solve a system involving the mass matrix during time stepping. 
The present paper builds on the work in \cite{hansbo2017cut}, where time-independent elasticity equations were solved using the Cut-FEM technique.

There are several reasons for using Cut-FEM to solve the elastic wave equation.
One example is when a boundary or interface has a complicated geometry.
Creating a computational mesh that conforms to this geometry can be expensive and time-consuming.
Using an immersed method could potentially be cheaper.
Another example is when the geometry of a boundary or interface is not known a priori.
This could, for example, be the case if the geometry of the interface is hard or impossible to measure.
One way to get around this is to send waves from the surface that propagate toward the interface and get reflected from it.
By measuring the reflected waves it is possible to solve an inverse problem to compute the shape of the interface.
In order to solve the inverse problem, one would need to iterate over a lot of different interface geometries.
Here, an immersed method would be useful since remeshing the interface geometry could be very time-consuming.
Similar inversion problems have been of interest for some time and were introduced partly by Tarantola in the papers \cite{tarantola1984inversion,tarantola1987inversion,tarantola1988theoretical}.

The present paper is organized as follows.
In Section~\ref{sec:model} the mathematical problems are stated.
These are the elastic wave equation posed on a single domain and as an interface problem.
This is followed by the explanation of the method in Section~\ref{sec:method}.
In Section~\ref{sec:theory} we present a proof of convergence for the single domain problem, and in Section~\ref{sec:experiments} we present numerical results on the order of convergence and robustness with respect to small cuts. Finally, we end with a discussion in Section~\ref{sec:discussion}. \section{Model of the Problem \label{sec:model}}
We are interested in the elastic wave equation posed both on a single domain ${\Omega\subset \R^d}$ (Figure~\ref{fig:single_domain}), and as an interface problem on a composite domain $\Omega={\Omega_1\cup \Omega_2} \subset \R^d$ (Figure~\ref{fig:interface_domain}).
The interface problem is interesting when we have two materials in contact with each other, which occurs frequently in applications due to the layered structure of the earth's crust.
On the other hand, the single domain problem is relevant if we have an inclusion of air or vacuum inside another material.

\begin{figure}[H]
\begin{subfigure}[b]{.3\paperwidth}
  \centering
  \includegraphics[width=.17\paperwidth]{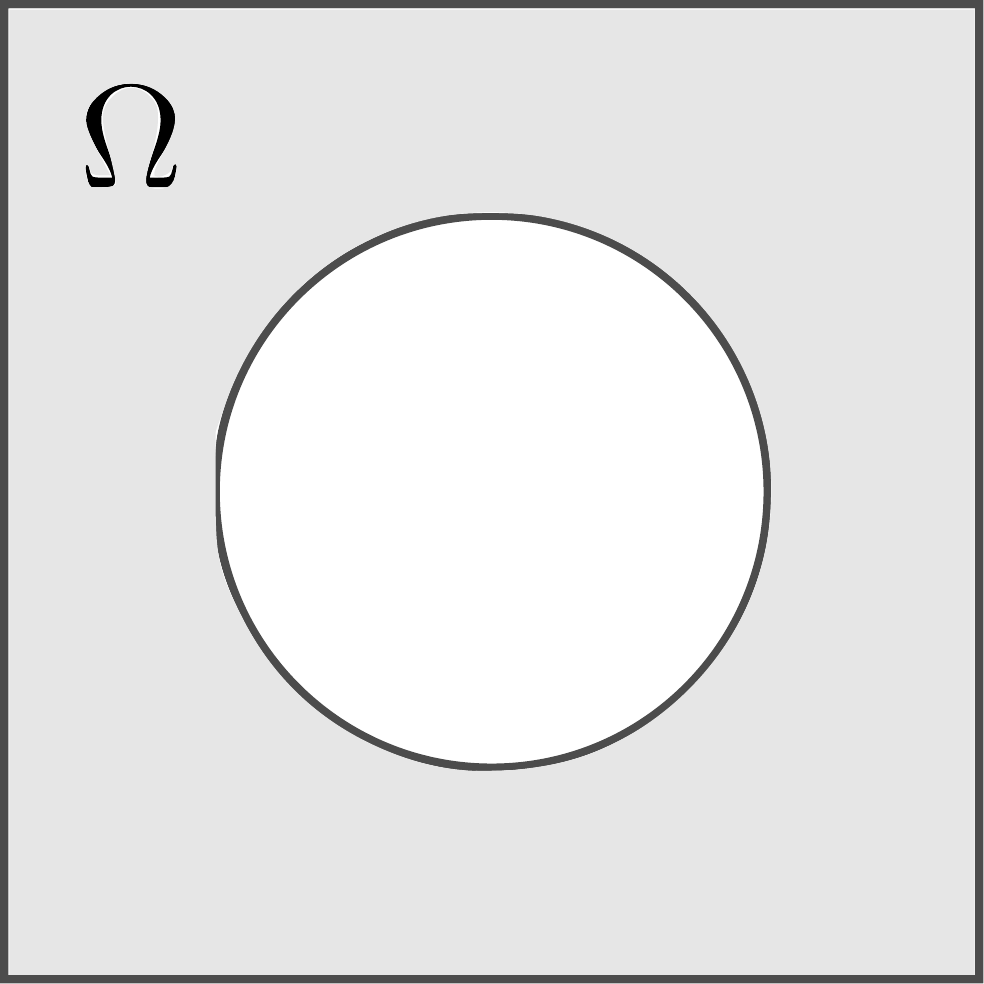}
  \caption{Single domain\label{fig:single_domain}}
\end{subfigure}
\begin{subfigure}[b]{.3\paperwidth}
  \centering
  \includegraphics[width=.17\paperwidth]{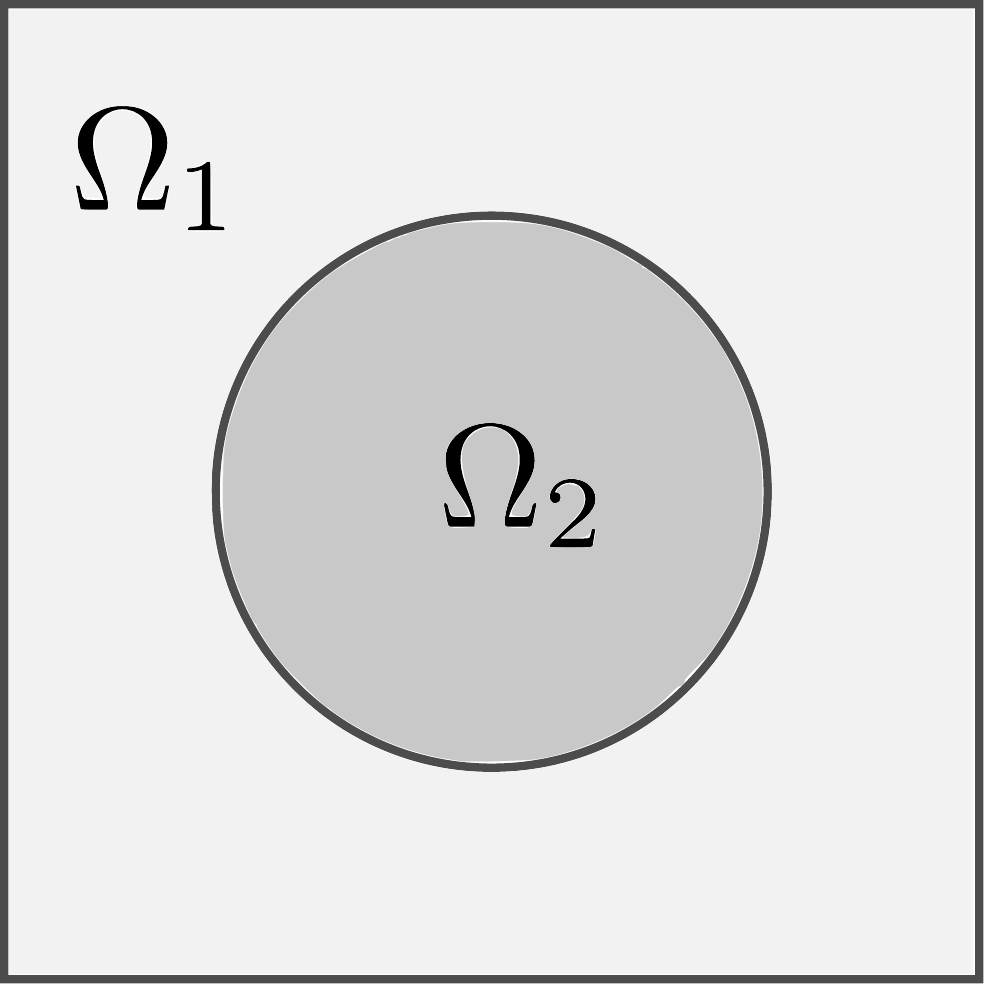}
  \caption{Composite domain for the interface problem\label{fig:interface_domain}}
\end{subfigure}
\caption{Considered domains\label{fig:domains}}
\end{figure}

\subsection{Single Domain Problem}\label{sec:elastic_single}
Let $n$ denote the outward unit normal to $\partial \Omega$, and assume that $\partial \Omega$ is partitioned such that $\partial \Omega=\Gamma^N\cup \Gamma^D$, with $\Gamma^N\cap \Gamma^D=\emptyset$.
The single domain problem reads:\begin{align}
\rho \ddot{u} &= \nabla \cdot \sigma (u) + f, \quad x\in \Omega, \label{eq:EWESingleDomain} \\ 
\sigma(u) \cdot n &= g^N, \quad x \in \Gamma^N, \label{eq:EWE_Neu} \\ 
u &= g^D,\quad x\in \Gamma^D, \label{eq:EWE_Dir} \\
u &= u^0, \quad t=0, \label{eq:EWE_Single_initial_u}\\
\dot{u} &= w^0, \label{eq:EWE_Single_initial_dudt} \quad t=0,
\end{align}
where $u$ is the displacement vector, $\rho$ is the density and $\sigma$ is the stress tensor.
We shall assume that $\Gamma^D$ and $\Gamma^N$ are sufficiently smooth.
Furthermore, we assume that we are working with a linear, homogeneous and isotropic material.
When this is the case the stress in the material is given by
\begin{equation}
\sigma_{ij}(u) = 2\mu \epsilon_{ij} (u) + \lambda (\nabla \cdot u)\delta_{ij},
\label{eq:Stress}
\end{equation}
where $\delta_{ij}$ is the Kronecker delta function and $\epsilon$ is the strain tensor defined as
\begin{equation} 
\epsilon_{ij}(u) = \frac{1}{2} \left(\pdd{u_i}{x_j} + \pdd{u_j}{x_i} \right).
\label{eq:Strain}
\end{equation}
In \eqref{eq:Stress} $\lambda$ and $\mu$ are the Lamé-parameters, which are material dependent scalar constants. 

\subsection{Interface Problem} \label{sec:elastic_elastic}
Consider now an interface problem on the domain illustrated in Figure~\ref{fig:interface_domain}.
We have a composite domain consisting of two elastic materials with material-parameters $\rho_i$, $\lambda_i$, $\mu_i$.
In this case, the problem is given by
\begin{align}
\rho_i \ddot{u}_i &= \nabla \cdot \sigma (u_i) + f_i, \quad x\in \Omega_i, \label{eq:EWECompositeDomain} \\ 
\sigma(u_i) \cdot n_i &= g^N_i, \quad x \in \Gamma^N_i, \label{eq:EWE_Neu_Composite} \\ 
u_i &= g^D_i,\quad x\in \Gamma^D_i, \label{eq:EWE_Dir_Composite} \\
\jump{u} &= 0, \quad x\in \Gamma_I, \label{eq:EWE_Interface_Composite_1} \\
\jump{\sigma(u) \cdot n} &= 0, \quad x\in \Gamma_I\label{eq:EWE_Interface_Composite_2},\\
u_i &= u_i^0, \quad t=0, \label{eq:EWE_Interface_initial_u}\\
\dot{u}_i &= w_i^0, \quad t=0, \label{eq:EWE_Interface_initial_dudt}
\end{align}
where $u_i$ is the displacement vector in material $i$ and the stress and strain tensors are defined analogously to \eqref{eq:Stress} and \eqref{eq:Strain}.
We assume that $\Gamma^D$, $\Gamma^N$ and $\Gamma^I$ are sufficiently smooth.
Here $n_i$ is the outward normal to $\Omega_i$ and $n$ is the normal pointing from $\Omega_2$ to $\Omega_1$ ($n=n_2$).
$\jump{\cdot} $ defines the jump over the interface:
\begin{equation}
\jump{u}=u_2(x)-u_1(x), \quad x\in \Gamma_I.
\label{eq:jump}
\end{equation}
Since we have several normals defined ($n_1$, $n_2$ and $n$), \eqref{eq:EWE_Interface_Composite_2} can be interpreted in two different ways.
To avoid any confusion we use the convention that the normal is fixed
\begin{equation}
\jump{\sigma(u) \cdot n} = \sigma(u_2) \cdot n-\sigma(u_1) \cdot n, \quad x\in \Gamma_I.
\end{equation}  \section{Numerical Method\label{sec:method}}
\noindent Let $\Omega$ be covered by a background mesh, $\T_B$, as in Figure~\ref{fig:aligned_immersed_parts}.
We shall only consider the case when the mesh consists of quadrilaterals that are squares and of the same size. 
Let $h$ denote their side length. 
Let the boundary or interface be partitioned as illustrated in Figure~\ref{fig:aligned_immersed_parts}.
That is, for the single domain we assume that $\partial \Omega=\Gamma^A\cup\Gamma^C$ (with $\Gamma^A\cap\Gamma^C=\emptyset$), where $\Gamma^A$ is aligned with the boundary of the mesh while $\Gamma^C$ cuts through it.
Correspondingly for the interface problem, we assume that $\partial \Omega \cup \Gamma_I=\Gamma^A\cup\Gamma^C$.
Let $\T^C$ denote the elements that are intersected by $\Gamma^C$:
\begin{equation}
\T^C=\{T\in \T_B : T\cap\Gamma^C\neq \emptyset\},
\end{equation}
as illustrated in Figure~\ref{fig:intersected_elements}.

Let $i\in\{1,2\}$ denote an index indicating the domain, which will be omitted for the single domain problem.
Let  $\T_i$, denote the smallest set of elements in the background mesh covering $\Omega_i$, as illustrated in Figure~\ref{fig:elements_solved_for}.
In particular, for the single domain problem, $\T$ is the smallest set of elements covering $\Omega$.
To be precise let
\begin{equation}
\T_i=\{T \in \T_B : T \cap \Omega_i\neq \emptyset \}.
\end{equation}

Now introduce the spaces
\begin{equation}
V_h^i={\{v\in [C^0(\Omega_i)]^d : \evaluated{v}{T} \in [Q_p(T)]^d, \, T \in \T_i \}}.
\end{equation}
Where $Q_p(T)$ denotes the $p$:th order Lagrange element with Gauss-Lobatto nodes over $T$.
For high element orders, Gauss-Lobatto nodes result in a mass matrix with better properties than if equidistant nodes are used \cite{karniadakis_spectral_hp_1999}.
For the single domain problem, we solve for the solution $u_h\in V_h$, while for the interface problem we solve for the pair $\pair{u}\in V_h^1\times V_h^2$.
For the interface problem, this means that the degrees of freedom are doubled over the elements in the set $\T^C$.

\begin{figure}
\begin{subfigure}[b]{.3\paperwidth}
\centering
  \includegraphics[height=.17\paperwidth]{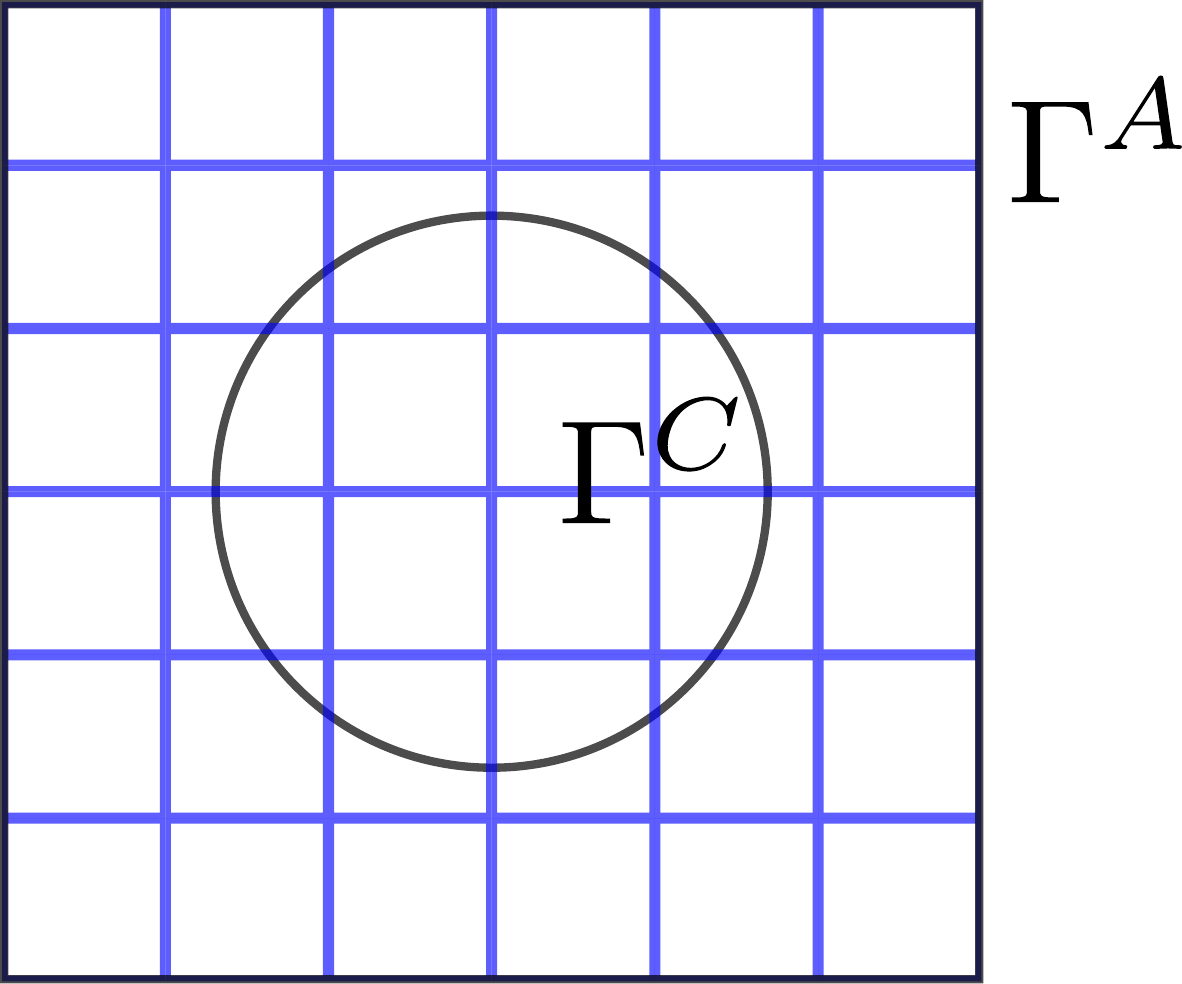}
  \caption{\label{fig:aligned_immersed_parts}}
\end{subfigure}
\begin{subfigure}[b]{.3\paperwidth}
  \centering
  \includegraphics[width=.17\paperwidth]{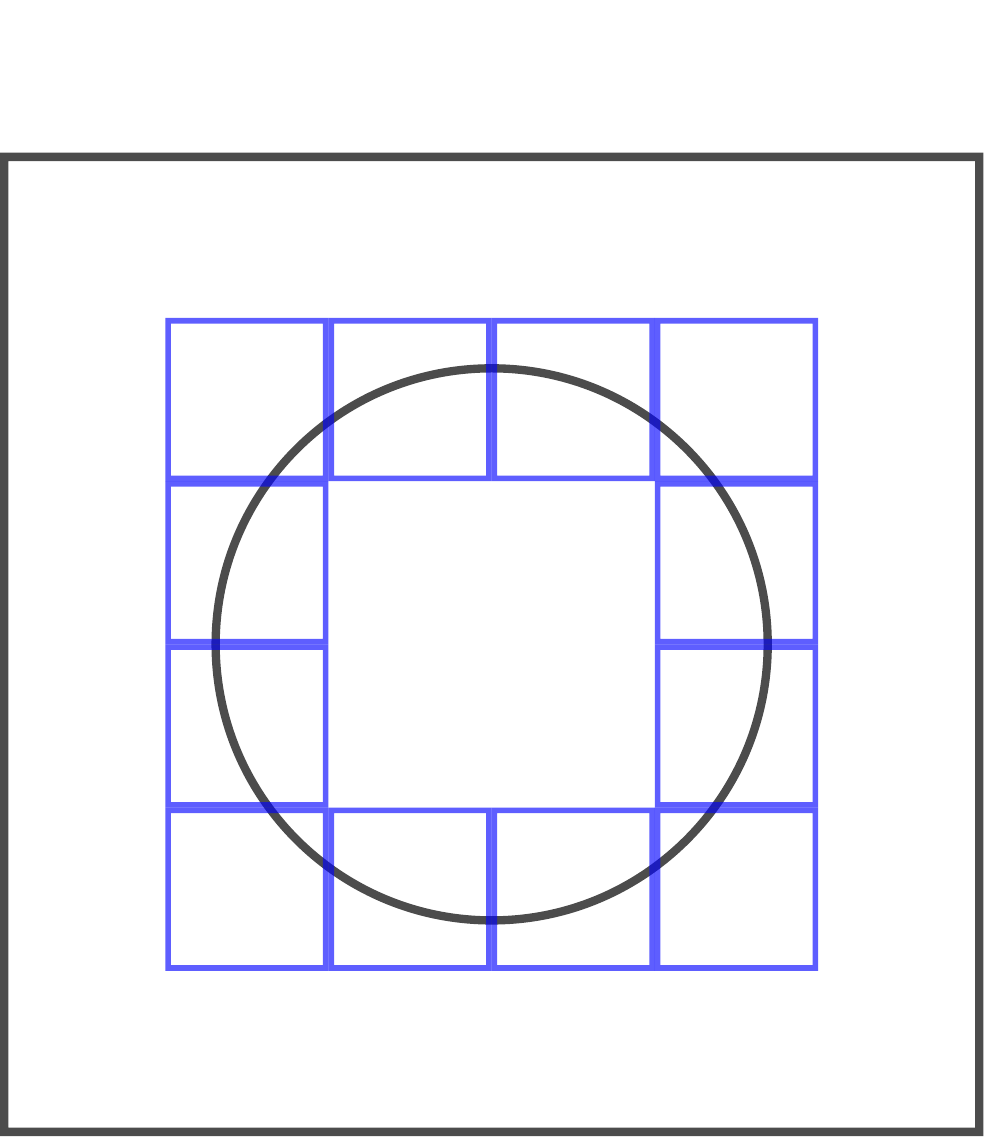}
  \caption{$\T^{C}$\label{fig:intersected_elements}}
\end{subfigure}
\caption{(a) Parts of the boundary or interface that are aligned with or immersed in the mesh.
(b) Set of elements, $\T^C$, intersected by $\Gamma^C$.}
\end{figure}

\begin{figure}
\begin{subfigure}[b]{.3\paperwidth}
  \centering
  \includegraphics[width=.17\paperwidth]{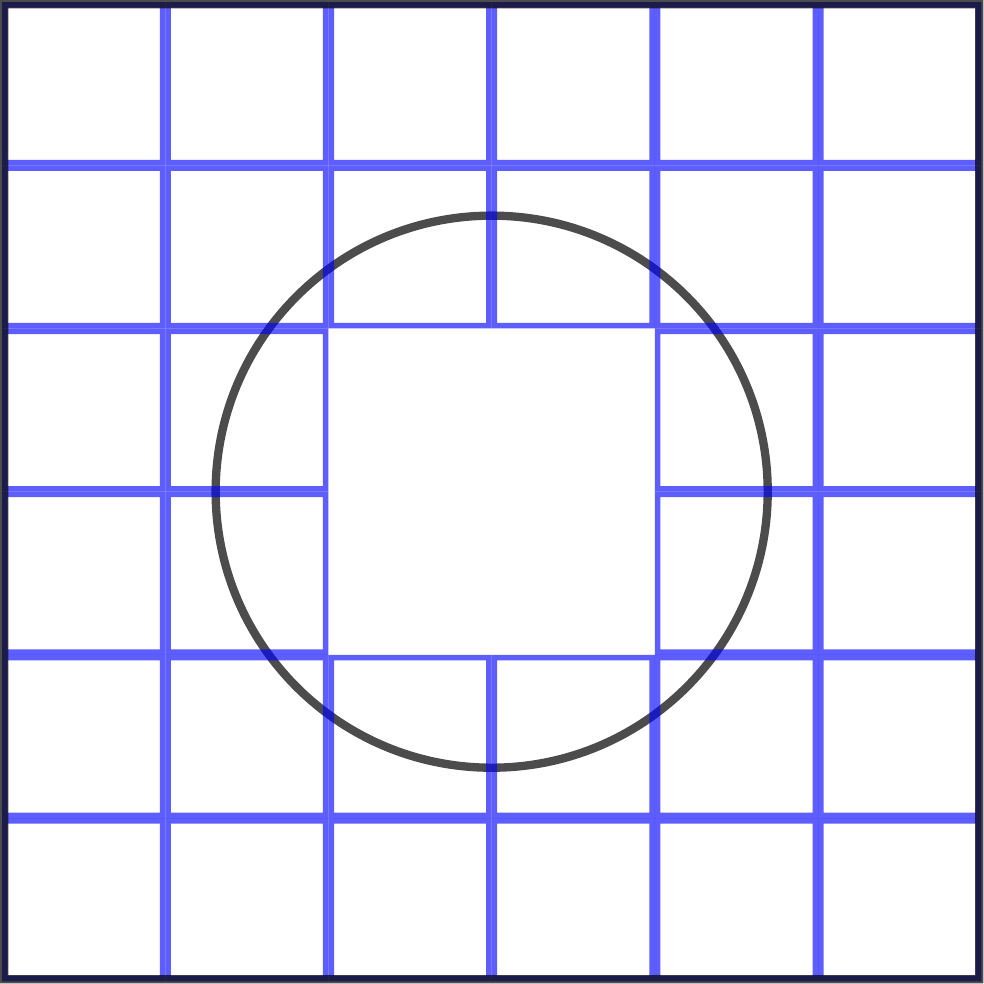}
  \caption{\label{fig:elements_solved_for_outside}}
\end{subfigure}
\begin{subfigure}[b]{.3\paperwidth}
  \centering
  \includegraphics[width=.17\paperwidth]{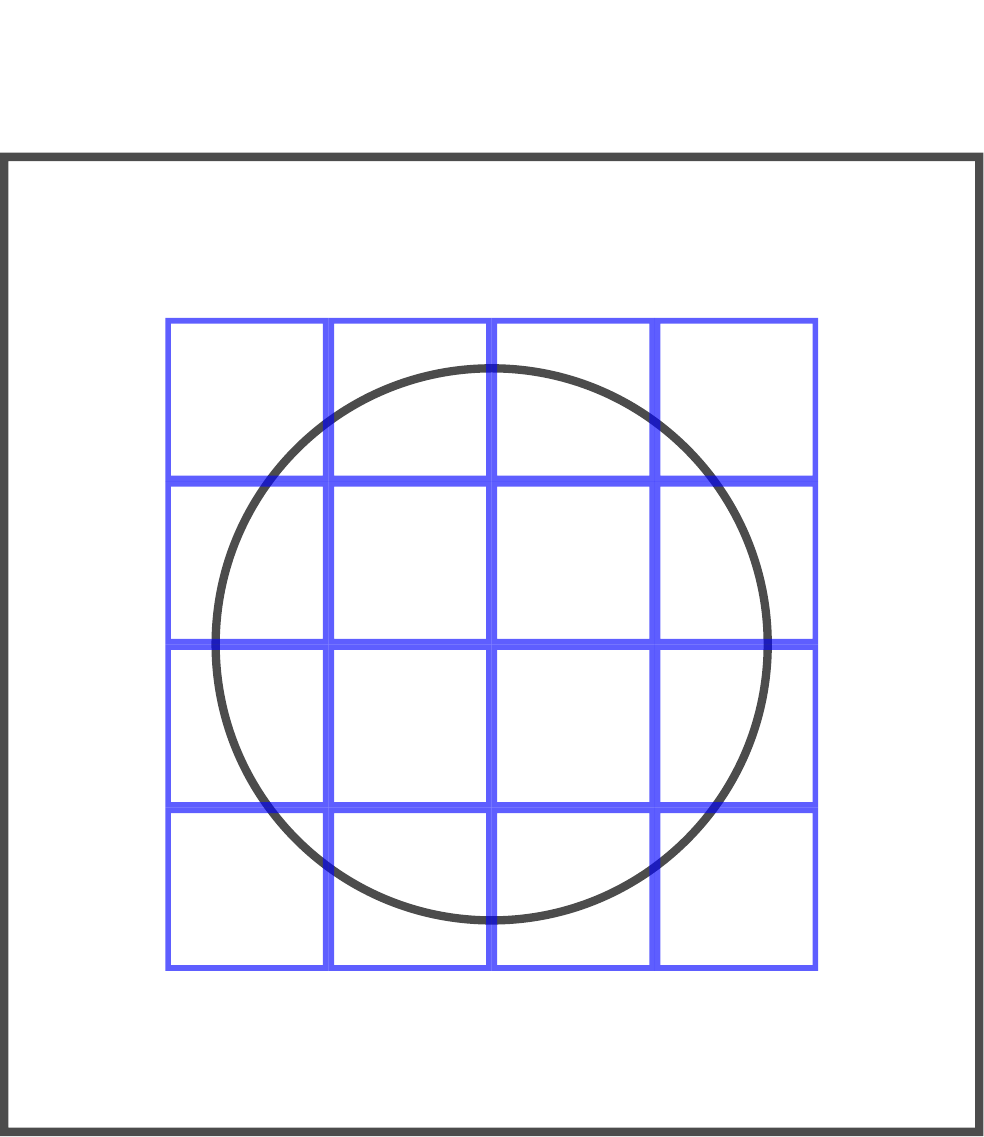}
  \caption{\label{fig:elements_solved_for_inside}}
\end{subfigure}
\caption{Smallest set of elements covering (a) $\Omega$, $\Omega_1$ and (b) $\Omega_2$ \label{fig:elements_solved_for}}
\end{figure}

Since the weak formulations for the single domain and the interface problem are very similar we discuss their derivation more or less simultaneously.
We shall use the following standard inner products
\begin{equation}
(u,v)_{\Omega}=\int_{\Omega} u v d\Omega, \quad \bsprod{u}{v}_{\Gamma}=\int_{\Gamma} u v d\Gamma,  
\label{eq:inner_products}
\end{equation}
where the subscripts indicate over which domain the integration takes part.
If $u$ or $v$ in \eqref{eq:inner_products} are tensors then contraction to a scalar is implied.
Note that the angular brackets denote integration over a curve in 2D (or surface in 3D).

By multiplying \eqref{eq:EWESingleDomain} or \eqref{eq:EWECompositeDomain} by a test function, integrating by parts and simplifying (for details see for example \cite{larson_fem_book}) we get
\begin{equation}
\begin{split}
(\rho_i\ddot{u}_i,v_i)_{\Omega_i} &
+2\mu_i (\epsilon(u_i), \epsilon(v_i))_{\Omega_i} 
+\lambda_i (\nabla \cdot u_i, \nabla \cdot v_i)_{\Omega_i} 
-\bsprod{\sigma(u_i)\cdot n_i}{v_i}_{\partial \Omega_i\setminus\Gamma^N_i} \\
&=(f_i,v_i)_{\Omega_i}
+ \bsprod{g^N_i}{v_i}_{\Gamma^N_i}, \quad \forall v_i\in V_h^i.
\end{split}
\label{eq:naiveWeakForm}
\end{equation}
Note that the Dirichlet boundary conditions are consistent with the following terms
\begin{align}
-\bsprod{u_i}{\sigma(v_i) \cdot n_i}_{\Gamma^D_i} &= -\bsprod{g^D_i}{\sigma(v_i)\cdot n_i}_{\Gamma^D_i}, \label{eq:dirichletConsistent1} \\ 
\lambda_i \frac{\gamma_D}{h} \bsprod{u_i}{v_i}_{\Gamma^D_i} &=\lambda_i \frac{\gamma_D}{h} \bsprod{g^D_i}{v_i}_{\Gamma^D_i}, \label{eq:dirichletConsistent2} \\
2\mu_i \frac{\gamma_D}{h} \bsprod{u_i\cdot n_i}{v_i\cdot n_i}_{\Gamma^D_i} &=  2\mu_i \frac{\gamma_D}{h} \bsprod{g^D_i\cdot n_i}{v_i\cdot n_i}_{\Gamma^D_i}.\label{eq:dirichletConsistent3}
\end{align}
So in order to enforce the boundary conditions by Nitsche's method we add \eqref{eq:dirichletConsistent1}--\eqref{eq:dirichletConsistent3} to \eqref{eq:naiveWeakForm}.
Here, $\gamma_D$ is a constant controlling how strongly the Dirichlet boundary condition is enforced.
We now have
\begin{equation}
(\rho_i\ddot{u}_i,v_i)_{\Omega_i}
+a_{i}(u_i,v_i) 
-\bsprod{\sigma(u_i)\cdot n_i}{v_i}_{\partial \Omega_i\setminus{(\Gamma^D_i \cup \Gamma^N_i)}} 
=L_i(v), \quad \forall v_i\in V_h^i,
\label{eq:weak_form_starting_point}
\end{equation}
where
\begin{equation} \label{eq:symmetric_weak_form}
a_i(u,v) = B_i(u,v) + D_i(u,v),
\end{equation}
and
\begin{equation}\label{eq:LinearFormSingleDomain}
L_i(v_i) = (f_i,v_i)_{\Omega_i} +\bsprod{g_N}{v_i}_{\Gamma^N_i} + L^D_{i}(v_i).
\end{equation}
In \eqref{eq:weak_form_starting_point} the term $B_i$ corresponds to integration over the ``bulk''
\begin{align}
B_{i}(u_i,v_i) &= 2\mu_i (\epsilon(u_i), \epsilon( v_i))_{\Omega_i} + \lambda_i (\nabla \cdot u_i, \nabla \cdot v_i)_{\Omega_i},
\label{eq:BiLinearBulkPart}
\end{align}
and the terms $D_i$ and $L^D_i$ enforce the Dirichlet boundary condition over $\Gamma^D_i$:
\begin{align}
D_{i}(u_i,v_i) =& 
-\bsprod{\sigma(u_i)\cdot n_i}{v_i}_{\Gamma^D_i} \nonumber \\
& -\bsprod{u_i}{\sigma(v_i)\cdot n_i}_{\Gamma^D_i}
 +\frac{\gamma_D}{h}\left(2\mu_i \bsprod{u_i}{v_i}_{\Gamma^D_i}
+\lambda_i \bsprod{u_i\cdot n_i}{v_i \cdot n_i}_{\Gamma^D_i}\right), \label{eq:BiLinearBoundaryPart}\\
L^D_{i}(v_i) 
=& - \bsprod{g^D_i}{\sigma(v_i)\cdot n_i}_{\Gamma^D_i}
+\frac{\gamma_D}{h} \left(2\mu_i \bsprod{g^D_i}{v_i}_{\Gamma^D_i} 
+ \lambda_i \bsprod{g^D_i\cdot n_i}{ v_i\cdot n_i}_{\Gamma^D_i}\right).\label{eq:LinearBoundaryPart}
\end{align}
Note that the terms \eqref{eq:dirichletConsistent1}--\eqref{eq:dirichletConsistent3} were added in a way so that $a_i$ in \eqref{eq:symmetric_weak_form} is a symmetric bilinear form.
Now \eqref{eq:weak_form_starting_point} is the starting point for the weak formulations for both the single domain and the interface problem.
Note also that for the single domain we have
\begin{equation*}
\partial\Omega\setminus(\Gamma^D \cup \Gamma^N)=\emptyset,
\end{equation*}
while for the interface problem
\begin{equation*}
\partial\Omega_i\setminus(\Gamma^D_i \cup \Gamma^N_i)=\Gamma_I.
\end{equation*}

\subsection{Stabilizing Small Cuts}
A common problem for immersed methods is robustness with respect to small cuts.
In order to understand this problem consider the single domain.
Since $\Gamma_C$ intersects the mesh in an arbitrary way an element $K$ may have an arbitrarily small intersection with the domain so that the size of $K\cap\Omega\ll h^d$.
For each element we integrate over $K\cap\Omega$.
For the mass matrix this means that the smallest eigenvalue can be arbitrarily small, and in turn that the condition number can be arbitrarily large.
For the stiffness matrix, the problem is even worse.
The term \eqref{eq:dirichletConsistent1} that we add to enforce the boundary condition can make some eigenvalues of the stiffness matrix negative, which would make the method unstable.

A suggested way to remedy this problem is to add a stabilizing term, $j_i$, both to the term that corresponds to the mass matrix and to the term that corresponds to the stiffness matrix:
\begin{equation} \label{eq:WeakFormM}
M_i(u_i,v_i) = (\rho_i u_i,v_i)_{\Omega_i} + \gamma_M^i j_i(u_i,v_i),
\end{equation}
\begin{equation} \label{eq:WeakFormA}
A_i(u_i,v_i) = a_{i}(u_i,v_i) + \frac{\gamma_A^i}{h^2}j_i(u_i,v_i).
\end{equation}
Here, $\gamma_M^i$ and $\gamma_A^i$ are scalar constants that control how much stabilization is added.
In order to explain the definition of $j_i$ let $\F_i$ denote the faces illustrated in Figure~\ref{fig:stabilized_faces}.
That is, the faces of $\T^C$ excluding the boundary faces of $\T_i$.
To be precise let
\begin{equation}
\F_i=\{F=T_a\cap T_b : T_a \in \T^C \text{ or } T_b \in \T^C,\quad T_a, T_b\in \T_i \}.
\end{equation}
We now define the stabilization term as
\begin{equation}
j_i(u,v)=\sum_{F\in \F_{i}} 
\sum_{k=1}^p \frac{h^{2k+1}}{(2k+1)(k!)^2}\bsprod{[\partial_n^k u_i]}{[\partial_n^k v_i]}_F.
\label{eq:jump_stabilization}
\end{equation}
Here, $\partial^k_n v_i$ denotes the $k$:th derivative in the direction of the face normal, $n$, and $[\cdot]$ defines the jump over a face $F$:
\begin{equation}
[u_i]=\evaluated{u_i}{F_{+}}-\evaluated{u_i}{F_{-}}.
\end{equation}
Note that $[\cdot]$ is different from $\jump{\cdot}$ in \eqref{eq:jump} since we have $u_i$ on both sides of $F$.
The stabilization in \eqref{eq:jump_stabilization} was suggested first in \cite{burman_ghost_2010} and used first for the Poisson equation in \cite{Burman2012}.
For a nice explanation of why it works see \cite{massing_stokes_2014}.

With stabilization one can prove the following inequalities for the bilinear form $M_i$ 
\begin{equation}
C_L\| v \|_{\Omega_i^\star}^2\leq M_i(v,v)\leq C_U \| v \|_{\Omega_i^\star}^2, \quad \forall v\in V_h^i,
\label{eq:mass_norm_equivalence}
\end{equation}
where $\Omega^\star_i$ is defined as the domain that $\T_i$ covers:
\begin{equation}
\Omega^\star_i=\bigcup_{T\in\T_i} T.
\end{equation}
In \eqref{eq:mass_norm_equivalence} $C_L$ and $C_U$ are positive constants that depend on the element order but not on $h$.
From \eqref{eq:mass_norm_equivalence} we immediately get that the eigenvalues of the stabilized mass matrix are bounded independently of how the boundary/interface cuts the mesh.
In turn, this bounds the condition number independently of the location of the boundary/interface.
Unfortunately (as noted in both \cite{sticko2016higher} and \cite{hansbo2017cut}) the constant in the bound increases very fast with the order of the elements.
With stabilization, one can also show that the bilinear form $A$ is continuous and coercive independently of how the boundary/interface cuts the mesh.
This result and the one in \eqref{eq:mass_norm_equivalence} were proved for the time-independent elasticity equations in \cite{hansbo2017cut}.

\begin{figure}
\begin{subfigure}[b]{.3\paperwidth}
  \centering
  \includegraphics[width=.17\paperwidth]{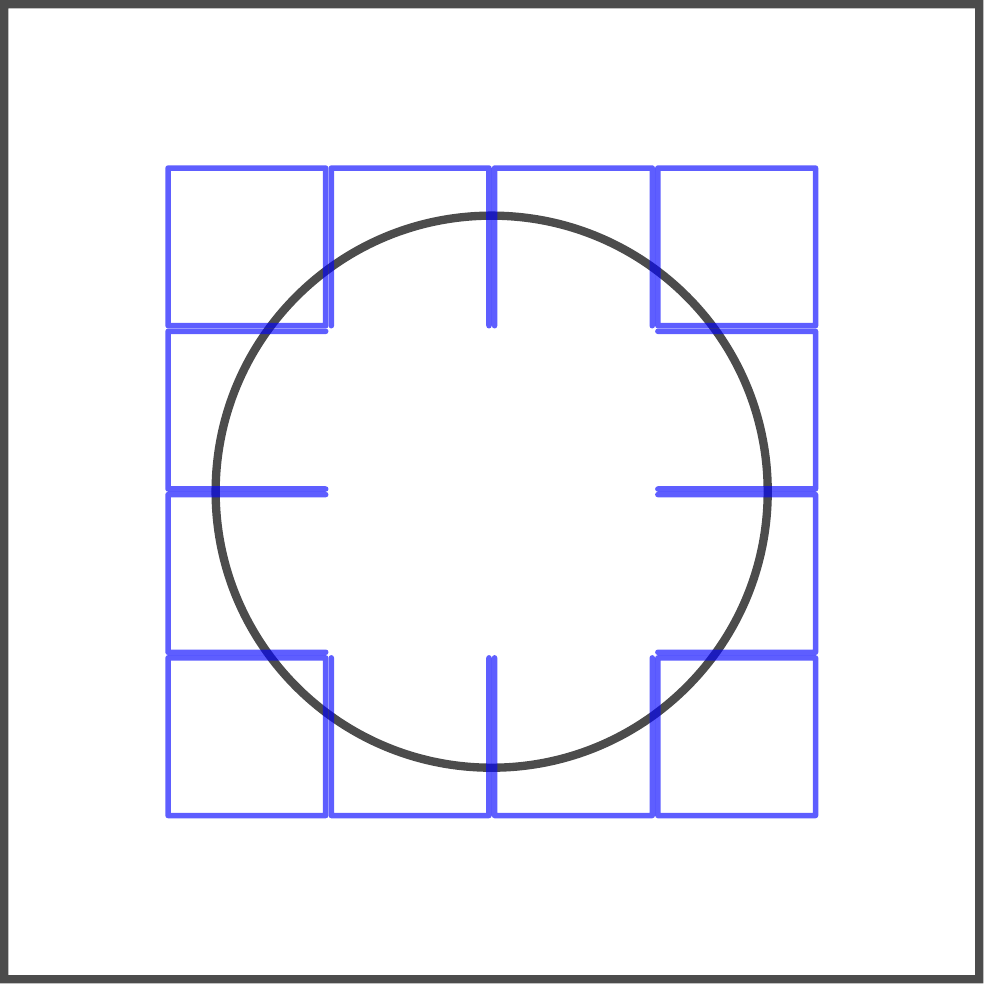}
  \caption{$\F$, $\F_1$\label{fig:stabilized_faces_outside}}
\end{subfigure}
\begin{subfigure}[b]{.3\paperwidth}
  \centering
  \includegraphics[width=.17\paperwidth]{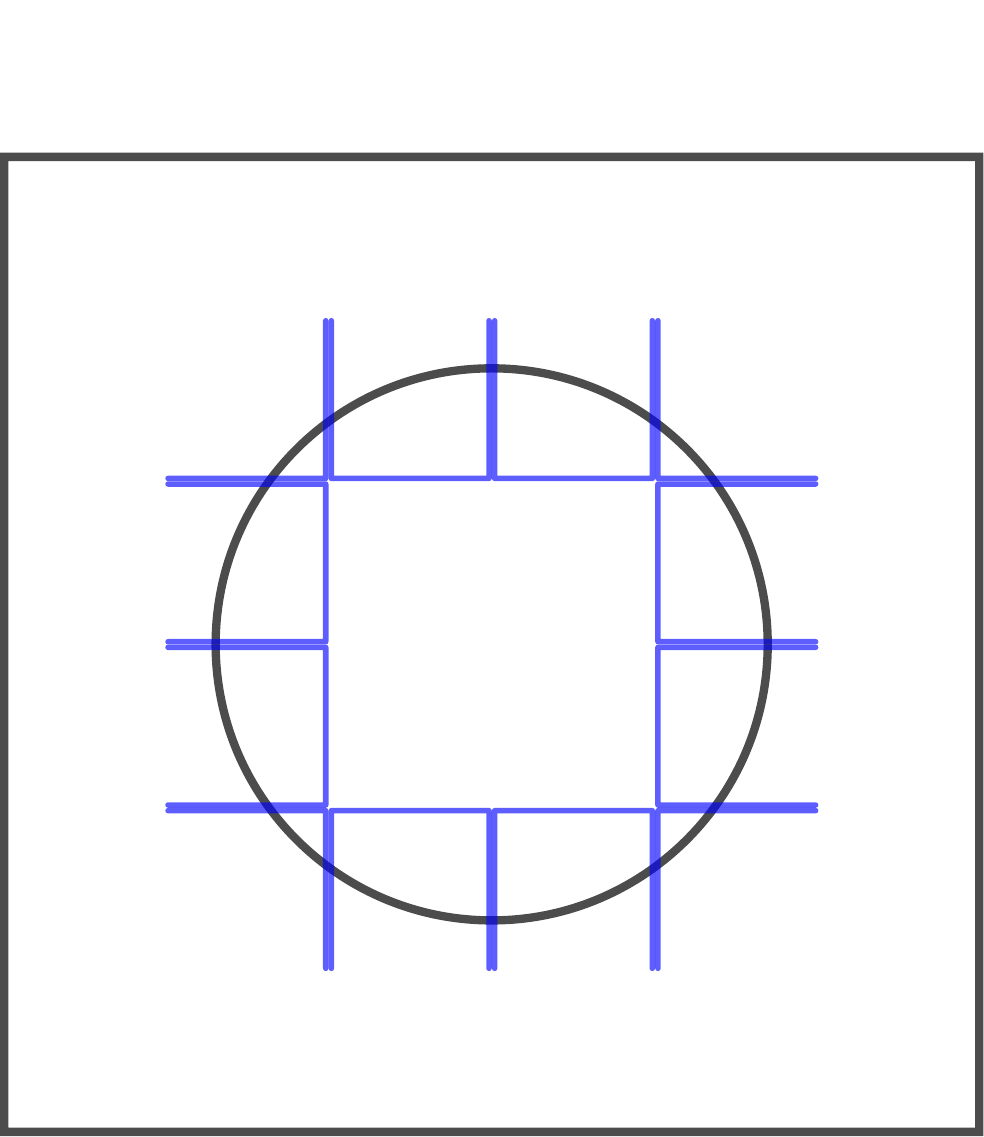}
  \caption{$\F_2$ \label{fig:stabilized_faces_inside}}
\end{subfigure}
\caption{Set of faces where the stabilization is applied \label{fig:stabilized_faces}}
\end{figure}

\subsection{Weak Form for the Single Domain Problem}
For the single domain we have that $\partial\Omega=\Gamma^D \cup \Gamma^N$, so by starting from \eqref{eq:weak_form_starting_point} and adding
the stabilizing terms we get the weak form for the single domain problem:
Find $u_h$ so that for each fixed $t\in (0, T]$, $u\in V_h$ such that 
\begin{equation}
M(\ddot{u},v)+A(u,v)=L(v), \quad \forall v \in V_h.
\label{eq:FEMethod_single}
\end{equation}  

\subsection{Weak Form for the Interface Problem}
We now want to derive the weak formulation for the interface problem \eqref{eq:EWECompositeDomain}--\eqref{eq:EWE_Interface_initial_dudt}.
First, let $\kappa_1>0$ and $\kappa_2>0$ fulfill $\kappa_1+\kappa_2=1$ and let $\{\cdot\}$ to denote the following convex combination:
\begin{equation}
\{v\}=\kappa_1 v_1+\kappa_2 v_2.
\end{equation}
By using that $\Omega_i\setminus(\Gamma^D_i \cup \Gamma^N_i)=\Gamma_I$, $n=n_2=-n_1$ and the condition \eqref{eq:EWE_Interface_Composite_2} it is straightforward to verify that 
\begin{align}
\sum_{i=1}^2 -\bsprod{\sigma(u_i)\cdot n_i}{v_i}_{\partial \Omega_i\setminus{(\Gamma^D_i \cup \Gamma^N_i)}}&=
\bsprod{\sigma(u_1)\cdot n}{v_1}_{\Gamma_I}-\bsprod{\sigma(u_2)\cdot n}{v_2}_{\Gamma_I} \nonumber\\
&=-\bsprod{\{ \sigma(u) \cdot n\}}{\jump{v}}_{\Gamma_I}.
\label{eq:first_convex_jump_term}
\end{align}
Note also that the interface condition \eqref{eq:EWE_Interface_Composite_1} is consistent with the following terms
\begin{align}
-\bsprod{\jump{u}}{\{\sigma(v)\cdot n\}}_{\Gamma_I}=0, \label{eq:penalty_interface_symmetry}\\ 
\frac{\gamma_I}{h}\bsprod{\jump{u}}{\jump{v}}_{\Gamma_I}=0 \label{eq:penalty_interface_positivity}.
\end{align}
Here, $\gamma_I$ is a positive constant which will control how strongly the interface condition is enforced.
Now we add \eqref{eq:weak_form_starting_point} for each domain, use \eqref{eq:first_convex_jump_term} and add \eqref{eq:penalty_interface_symmetry}, \eqref{eq:penalty_interface_positivity} and stabilization to obtain the finite element method: 
Find $u=\pair{u}$ so that for each fixed $t\in (0, T]$, $u\in V_h^1\times V_h^2$ such that 
\begin{equation}
\sum_{i=1}^2 \left( M_i(\ddot{u}_i,v_i)+A_i(u_i,v_i) \right )+I(u,v)=\sum_{i=1}^2 L_i(v_i), \quad \forall v=\pair{v}\in V_h^1\times V_h^2.
\label{eq:FEMethod_interface}
\end{equation}
Here $M_i$, $A_i$ and $L_i$ were defined in \eqref{eq:WeakFormM}, \eqref{eq:LinearFormSingleDomain} and \eqref{eq:BiLinearBulkPart}. 
The bilinear form $I$ that enforce the interface conditions is given by 
\begin{equation}\label{eq:BiLinearFormInterfacePart}
I(u,v) = -\bsprod{\{ \sigma(u) \cdot n\}}{\jump{v}}_{\Gamma_I} 
-\bsprod{\jump{u}}{\{\sigma(v) \cdot n\}}_{\Gamma_I} 
+\frac{\gamma_I}{h} \bsprod{\jump{u}}{\jump{v}}_{\Gamma_I}.
\end{equation}

The method contains a number of free parameters that need to be chosen.
Clearly, the penalty parameters related to the stabilization should scale with the parameters of the materials.
We choose to scale them as
\begin{equation}
\gamma_M^i= \frac{1}{4}\rho_i, \quad \gamma_A^i=\frac{1}{2}\eta_i,
\label{eq:stabilization_parameters}
\end{equation}
where
\begin{equation}
\eta_i=2\mu_i+\lambda_i.
\end{equation}
We choose the constants related to the interface terms in the following way
\begin{equation}
\kappa_1 = \frac{\idx{\eta}{2}}{\idx{\eta}{1} + \idx{\eta}{2}}, \quad 
\kappa_2 = \frac{\idx{\eta}{1}}{\idx{\eta}{1} + \idx{\eta}{2}}, \quad
\gamma_{I} =  20p^2\frac{\idx{\eta}{1}\idx{\eta}{2}}{\idx{\eta}{1} + \idx{\eta}{2}}. \label{eq:kappas_convex_combination}
\end{equation}
The scaling with respect to $\eta_i$ is analogous to the choice of parameters for the Poisson interface problem in \cite{CutFEM2014}.
The Nitsche parameter related to the Dirichlet boundary condition is chosen as
\begin{equation}
\gamma_D=5p^2.
\label{eq:Nitsche_parameter}
\end{equation}
Here, the scaling with $p^2$ of $\gamma_D$ and $\gamma_I$ follows from an inverse inequality.
The numerical constants are chosen based on experience.
We shall briefly discuss this in Section~\ref{sec:discussion}.
\subsection{Imposition of Initial Conditions \label{sec:initial_condition}}
In order to impose the initial conditions we first define the stabilised $L_2$-projection, $\Pi_h u$.
For the single domain problem, $\Pi_h u$ is defined as the solution to the following problem: Given $u$, find $\Pi_h u \in V_h$ such that
\begin{equation}
M(\Pi_h u,v)=(u,v)_\Omega, \quad \forall v\in V_h.
\label{eq:L2_projection}
\end{equation}
For the interface problem, $\Pi_h u$ is defined analogously as the solution to:
Given $u$, find $\Pi_h u=\pair{\Pi_hu}\in V_h^1\times V_h^2$ such that
\begin{equation}
\sum_{i=1}^2 M_i(\Pi_h u_i,v_i)=\sum_{i=1}^2 (u,v_i)_\Omega, \quad \forall v=\pair{v}\in V_h^1\times V_h^2.
\label{eq:L2_projection_interface}
\end{equation}
The initial conditions are now imposed as
\begin{equation}
\begin{aligned}
u_h |_{t=0} &= \Pi_h u|_{t=0}, \\
\dot{u}_h |_{t=0} & = \Pi_h \dot{u}|_{t=0}.
\end{aligned}
\label{eq:initial_conditions}
\end{equation}
Note that, by setting the discrete initial conditions in this way, the initial conditions of the single domain problem, \eqref{eq:EWE_Single_initial_u}--\eqref{eq:EWE_Single_initial_dudt}, only need to be defined on $\Omega$ and not on $\Omega^\star$. 
 \section{Theory\label{sec:theory}}
In this section, we will present some theoretical results, in particular, a proof of convergence for the semi-discrete method for the single domain problem.
The proof builds on the results presented in \cite{hansbo2017cut} where several time-independent problems were studied.
During the analysis we will use the following norms:
\begin{equation}  
\|v\|^2_{M} = M(v,v), \quad \|v\|^2_{A} = A(v,v), \quad |v|^2_{j} = j(v,v),
\end{equation}
\begin{equation}
\en{v}^2_h = \|v\|_A^2 + h\|\sigma(v)\|^2_{\Gamma^D} + \frac{1}{h}(2\mu  \|v\|^2_{\Gamma^D} + \lambda \| v\cdot n \|^2_{\Gamma^D}),
\label{eq:tripple_norm_def}
\end{equation}
where we can note that $|\cdot |_j$ is a semi-norm. 
Note that these norms only make sense if the argument is defined on $\Omega^*$.
We will also use the $\lesssim$ -- relation, which we define as 
\begin{equation}
a\lesssim b \Leftrightarrow a \leq C b,
\end{equation}
where $C$ is some constant that is independent of $h$.

We also need a bounded extension operator, $\extension: H^{s}(\Omega)\rightarrow H^{s}(\Omega^\star)$.
We shall assume that the solution is sufficiently smooth ($s$ is sufficiently high) and that $\partial \Omega$ is sufficiently regular so that
\begin{equation}
j(\extension\ddot{u},v) = 0, \quad \forall v\in V_h.
\label{eq:j_of_analytical_is_zero}
\end{equation} 

\subsection{Ritz Projection}
In order to prove convergence we need a ``Ritz-like'' projection, which we define as the solution to the following problem:
Given $u$, find $R_h u \in V_h$ such that
\begin{equation}
A(R_h u,v)=a(u,v), \quad \forall v\in V_h.
\label{eq:Ritz_projection}
\end{equation}
In this section, we will gather some results about the Ritz projection, which will be essential in the analysis to come.
For brevity, we will from here on omit the ``like'' in the Ritz-like projection \eqref{eq:Ritz_projection} and simply call it the Ritz projection.
As shown in \cite{hansbo2017cut}, given that $\gamma_D$ is sufficiently large, $A$ is coercive and continuous with respect to $\en{\cdot}_h$.
That is, there exists constants $C_r, C_c>0$ such that
\begin{equation}
C_r\en{v_h}_h^2 \leq A(v_h,v_h), \quad A(v_h,w_h) \leq C_c \en{v_h}_h\en{w_h}_h, \quad v_h,w_h\in V_h. 
\label{eq:coercive_continuous}
\end{equation}
For simplicity, we will assume that $\Gamma^D\neq \emptyset$.
When this holds, $\en{\cdot}_h$ is indeed a norm (i.e. not only a semi-norm) and \eqref{eq:Ritz_projection} has a unique solution.
However, this assumption can likely be relaxed by looking for the solution in a constrained subspace of $V_h$.

One should note that this projection is nothing but the solution to the time-independent elasticity problem.
To see this, let $\hat{u}(x) = u(x,t_{f})$, where $t_{f}$ is some fixed time, and define $\hat{f}$ so that $\hat{u}$ is the solution to
\begin{equation}
\begin{aligned}
&\nabla \cdot \sigma(\hat{u}) = - \hat{f}, \quad x\in\Omega, \\
&\hat{u} = g^D(x,t_{f}), \quad x \in \Gamma^D, \\
&\pdd{\hat{u}}{n} = g^N(x,t_{f}), \quad x\in \Gamma^N.
\end{aligned}
\label{eq:Ritz_problem}
\end{equation}
This means that  $\hat{u}$ will satisfy
\begin{equation} \label{eq:Ritz_VF}
a(\hat{u},v) = \hat{L}(v),
\end{equation}
where $\hat{L}$ is defined as 
\begin{equation*}
\begin{aligned}
\hat{L}(v) = &(\hat{f},v) + \bsprod{g^N}{v}_{\Gamma_N} -\bsprod{g^D}{\sigma(v)\cdot n}_{\Gamma^D} + \\ &\frac{\gamma_D}{h} \left(2 \mu \bsprod{g^D}{v}_{\Gamma^D} + \lambda \bsprod{g^D\cdot n}{v \cdot n}_{\Gamma^D}\right),
\end{aligned}
\end{equation*}
i.e. the same as $L$ in \eqref{eq:LinearFormSingleDomain} but using the right hand side data from \eqref{eq:Ritz_problem}.
We can now formulate the finite element method to solve \eqref{eq:Ritz_problem} as: Find $\hat{u}_h \in V_h$ such that
\begin{equation} \label{eq:Ritz_FEM}
A(\hat{u}_h,v_h) = \hat{L}(v_h), \quad \forall  v_h \in V_h.
\end{equation}
Now, by subtracting \eqref{eq:Ritz_VF} from \eqref{eq:Ritz_FEM} we can see that the solution $\hat{u}$, to the problem \eqref{eq:Ritz_problem}, in fact corresponds to the Ritz projection $R_h u$ in \eqref{eq:Ritz_projection}.
So in principle the Ritz projection is obtained by solving a linear elasticity problem. This has been treated in detail in \cite{hansbo2017cut}, where the results presented in Lemma~\ref{lem:Ritz_bound} were derived.
\begin{lemma} \label{lem:Ritz_bound}
For the Ritz projection, $R_hu$, in \eqref{eq:Ritz_projection} the following error estimates hold
\begin{equation} \label{eq:Ritz_energy_bound}
\en{R_h u - \extension u}_h \lesssim h^k  \|u\|_{H^{k+1}(\Omega)},
\end{equation}
\begin{equation}\label{eq:Ritz_bound}
\|R_h u - u\|_{\Omega} \lesssim h ^{k+1} \|u\|_{H^{k+1}(\Omega)}.
\end{equation}
\end{lemma}
\begin{proof}
See Theorem~4.2 in \cite{hansbo2017cut}.
\qed
\end{proof}
\newpage
\noindent
We shall also need the following corollary.
\begin{corollary}
For the Ritz projection, $R_hu$, in \eqref{eq:Ritz_projection} the following holds
\begin{equation}
|\extension u-R_h u |_j \lesssim h^{k+1} \|u \|_{H^{k+1}(\Omega)}.
\label{eq:j_bound}
\end{equation}
\label{cor:j_bound}
\end{corollary}
\begin{proof}
From \eqref{eq:Ritz_energy_bound} and the definition of $\en{\cdot}_h$ in \eqref{eq:tripple_norm_def} we get
\begin{equation*}
h^{-2}|\extension u-R_h u |_j^2 \lesssim \en{R_h u - \extension u}_h^2 \lesssim h^{2k}  \|u\|_{H^{k+1}(\Omega)}^2,
\end{equation*}
from which \eqref{eq:j_bound} follows.
\qed
\end{proof}

\subsection{A priori Analysis}
The analysis presented here is similar to the one presented in \cite{sticko2016}.
 We wish to bound the error $u_h -u$ and in doing so we split the error in two parts,
\begin{equation}\label{eq:the_split}
u_h - \extension u = e_N  + e_R
\end{equation}
where $e_N = u_h - R_h u$ and $e_R = R_hu - \extension u.$
By Lemma~\ref{lem:Ritz_bound} we directly get a bound for $e_R$. 
To bound $e_N$ we first aim to find a bound on the ``energy'' of $e_N$, which we define as 
\begin{equation}
E_{e_N} = \frac{1}{2} (M(\dot{e}_N,\dot{e}_N) + A( e_N, e_N)).
\label{eq:energy}
\end{equation}
To facilitate the proof, we will in this section assume that the discrete initial conditions are imposed using the Ritz-projection:
\begin{equation}
\begin{aligned}
u_h |_{t=0} &= R_h u|_{t=0}, \\
\dot{u}_h |_{t=0} & = R_h \dot{u}|_{t=0}.
\end{aligned}
\label{eq:Ritz_as_initial}
\end{equation}
Note that \eqref{eq:Ritz_as_initial} is not the same initial conditions as in \eqref{eq:initial_conditions}, which are used in the numerical experiments.
The reason for this is that computing the Ritz-projection is more involved than computing the $L_2$-projection.
In practice, this most likely makes no difference since the result of both projections approximates the analytical solution with the same order of accuracy.
However, the choice \eqref{eq:Ritz_as_initial} makes the analysis simpler since it is equivalent to
\begin{align}
e_N|_{t=0}=0, \label{eq:eN_zero_initially}\\
\dot{e}_N|_{t=0}=0, 
\end{align}
which by the definition of the energy in \eqref{eq:energy} gives us
\begin{equation}
E_{e_N}|_{t=0} = 0.
\label{eq:EnergyZeroInitially}
\end{equation}
We are now ready to bound the energy.
\newpage
\noindent
\begin{lemma}
The following bound holds
\begin{equation}\label{eq:E_bound}
E_{e_N}(t) \lesssim h^{2(k+1)}.
\end{equation}
\label{lem:energy_bound}
\end{lemma}
\begin{proof}
First, we have that
\begin{equation}
\begin{aligned}
M(\ddot{e}_N, v_h) + A( e_N, v_h) &= M(\ddot{u}_h,v_h) + A(u_h,v_h) - M(R_h \ddot{u},v_h) - A(R_hu, v_h) \\
&=(\ddot{u},v_h)_\Omega + a(u,v_h) - M(R_h \ddot{u},v_h) - A(R_hu, v_h) \\
&=(\ddot{u},v_h)_\Omega - M(R_h \ddot{u},v_h)\\
&=(\ddot{u},v_h)_\Omega - M(R_h \ddot{u},v_h)+ \gamma_M j(\extension\ddot{u}, v_h)\\
&= M(-\ddot{e}_R,v_h),
\end{aligned}
\label{eq:first_step_energy_bound}
\end{equation}
where we in the first line used the definition of $e_N$.
When going to the second line we used the definition of the finite element method in \eqref{eq:FEMethod_single} and that the analytical solution satisfies
\begin{equation*}
(\ddot{u},v_h)_\Omega+a(u,v_h)=L(v_h), \quad \forall v_h\in V_h.
\end{equation*}
When going to the third line we used the definition of the Ritz projection in \eqref{eq:Ritz_projection}.
Finally we used \eqref{eq:j_of_analytical_is_zero} and the definition of $e_R$.
Now, choosing $v_h = \dot{e}_N$ in \eqref{eq:first_step_energy_bound} we can use the definition of the energy and that $M$ is an inner product (so that Cauchy-Schwarz applies) to get
\begin{equation}
\dd{E_{e_N}}{t} \leq \|\ddot{e}_R\|_M \|\dot{e}_N\|_M \leq \|\ddot{e}_R\|_M \sqrt{2E_{e_N}}.
\label{eq:derivative_of_energy_bound}
\end{equation}
By using
\begin{equation*}
\dd{E_{e_N}}{t}=\dd{}{t} (\sqrt{E_{e_N}})^2 =2\sqrt{E_{e_N}}\,\dd{}{t}\sqrt{E_{e_N}},
\end{equation*}
we can divide both sides of \eqref{eq:derivative_of_energy_bound} by $2\sqrt{E_{e_N}}$ and get
\begin{equation}\label{eq:sqrt_E_bound}
\begin{aligned}
\dd{}{t}\sqrt{E_{e_N}} &\leq\frac{1}{\sqrt{2}} \|\ddot{e}_R\|_M\\
& \leq\frac{1}{\sqrt{2}}\sqrt{\|\ddot{e}_R\|_{\Omega}^2 + \gamma_M |\extension\ddot{u}-R_h\ddot{u}|_j^2}\\
& \leq C h^{k+1}\|\ddot{u}\|_{H^{k+1}(\Omega)},
\end{aligned}
\end{equation}
where we in the last line used Lemma~\ref{lem:Ritz_bound} and Corollary~\ref{cor:j_bound}.
Integrating and squaring \eqref{eq:sqrt_E_bound} gives
\begin{equation}
E_{e_N}(t)\leq \left( \sqrt{E_{e_N}(0)}+Ch^{k+1} \int_0^t \|\ddot{u}\|_{H^{k+1}(\Omega)} d\tau \right)^2.
\label{eq:energy_bound_as_integral}
\end{equation}
Finally, using \eqref{eq:EnergyZeroInitially} gives us the bound in \eqref{eq:E_bound}.
\qed
\end{proof}
We are now ready to state our a priori error estimates.
They are summed up in Theorem~\ref{thm:A_priori}.
\begin{theorem}\label{thm:A_priori}
Let $u$ be the solution to \eqref{eq:EWESingleDomain}--\eqref{eq:EWE_Single_initial_dudt} and let $u_h$ be the solution to \eqref{eq:FEMethod_single}, then at any given time, $t$, the following a priori error estimates hold
\begin{align}
&\|u_h - u \|_{\Omega} \lesssim h^{k+1} ,\label{eq:u_error} \\
&\|\nabla u_h - \nabla u \|_{\Omega} \lesssim h^k .\label{eq:grad_error}
\end{align}
\end{theorem}
\begin{proof}
Using the definition of $E_{e_N}$ and Lemma~\ref{lem:energy_bound} we get 
\begin{alignat}{2}
\|\dot{e}_N \|_\Omega &=\|\dot{u}_h(t) - R_h\dot{u}(t) \|_\Omega && \lesssim h^{k+1}, \label{eq:bound1} \\
\|e_N \|_A &=\|u_h(t) - R_hu(t)\|_A &&\lesssim h^{k+1}. \label{eq:bound2}
\end{alignat}		 
In order to bound $e_N$ and not $\dot{e}_N$ note that 
\begin{equation}
2\|e_N\|_\Omega \dd{}{t}\|e_N\|_\Omega = \dd{}{t}\|e_N\|^2_\Omega = 2(\e_N, \dot{e}_N)_\Omega \leq 2 \|\e_N\|_{\Omega} \|\dot{e}_N\|_{\Omega}.
\label{eq:bound_derivative_eN}
\end{equation}
Dividing \eqref{eq:bound_derivative_eN} by $2\|e_N\|_\Omega$ and integrating over time gives
\begin{equation} \label{eq:derivative_bound}
\|e_N(t) \|_\Omega \leq \int_0^t \|\dot{e}_N(\tau)\|_{\Omega} \, d\tau,
\end{equation}
by using \eqref{eq:eN_zero_initially}.
Combining \eqref{eq:derivative_bound} with \eqref{eq:bound1} gives us
\begin{equation} \label{eq:bound_3}
\|u_h(t) - R_hu(t) \|_\Omega \lesssim h^{k+1}.
\end{equation}
Finally, we use the triangle inequality on \eqref{eq:the_split} and combine \eqref{eq:bound2} and \eqref{eq:bound_3} with the bounds on $e_R$ from Lemma~\ref{lem:Ritz_bound} to get the estimates \eqref{eq:u_error} and \eqref{eq:grad_error}.
\qed
\end{proof}
\subsection{Time Step Restriction} \label{sec:time_step_restriction}
Both of the weak forms \eqref{eq:FEMethod_single} and \eqref{eq:FEMethod_interface} will discretize to a system of the form
\begin{equation}
\mathcal{M} \ddot{\xi}+ \mathcal{A}\xi=\mathcal{L}(t),
\label{eq:discrete_system}
\end{equation}
where $\mathcal{M}\in\R^{N\times N}$ is the mass-matrix, $\mathcal{A}\in\R^{N\times N}$ is the stiffness-matrix and $\mathcal{L}\in\R^{N}$ is the right-hand side vector.

If we use explicit time stepping the largest time step, $\tau$, we can take due to stability restrictions will be bounded by the $C_{FL}$-number as:
\begin{equation}
\tau \leq \alpha C_{FL} h,
\end{equation}
where $\alpha$ is a constant which depends on the chosen time stepping scheme.
The $C_{FL}$-number can be computed from the matrices in the discrete system.
Let $\lambda_{\max}$ be the largest eigenvalue of the generalized eigenvalue problem: find $\lambda$, $x\in \R^N$ such that
\begin{equation}
\mathcal{A} x- \lambda \mathcal{M}x=0.
\end{equation}
Then the $C_{FL}$-number is given by
\begin{equation}
C_{FL}=\frac{1}{h\sqrt{\lambda_{\max}}}.
\label{eq:cfl_from_matrices}
\end{equation}
It is important that the $C_{FL}$-number does not decrease significantly when the smallest cut in the mesh approaches zero.
Ideally, the time step restriction should not be more severe than for the standard non-cut finite element method.

\subsection{Material Parameters}
The problem for the single domain contains three material parameters, $\rho$, $\lambda$ and $\mu$.
However, by rescaling (see \cite{langtangen_scaling_2016}) one can show that the dimensionless equation only depends on the ratio, $\beta$, between the Lamé-parameters:
\begin{equation}
\beta=\frac{\lambda}{\mu}.
\end{equation}
Thus we can without loss of generality assume that the equation is already in dimensionless form and set $\rho=\mu=1$.
Now we can obtain different physical behavior by varying $\lambda$.
For the interface problem, we shall also assume that we are working in dimensionless form.
By a corresponding analysis, it is possible to show that we can set $\rho_1=\mu_2=1$ and obtain different physical behavior by varying $\lambda_1$, $\rho_2$, $\lambda_2$ and $\mu_2$. \section{Numerical Experiments\label{sec:experiments}}
In this section, we present some numerical examples.
First, we investigate if the error converges with the expected order.
This is done for the single domain problem in Section~\ref{sec:experiment_single_convergence} and for the interface problem in Section~\ref{sec:experiment_interface_convergence}. 
In Section~\ref{sec:experiment_decreasing_cut} we investigate how the properties of the discretized matrices in \eqref{eq:discrete_system} change when the smallest cut in the mesh approaches zero.
To implement the method, we have used the finite element library deal.II \cite{dealII}.
A level set function has been used to represent the immersed boundary/interface.
To generate high order quadrature rules on the intersected elements we have used the algorithm from \cite{Saye2015}.

In the experiments below the following material parameters have been used
\begin{align}
&\rho=\rho_1=1,           \quad &&\rho_2=1.1154, \nonumber\\
&\lambda=\lambda_1=1.1429,\quad &&\lambda_2=2.6182, \label{eq:material parameters}\\
&\mu=\mu_1=1,             \quad &&\mu_2=1.8.\nonumber
\end{align}
These parameters correspond to material 1 being sandstone and material 2 being granite, these are two of the most common rock types.
Note that by using the present model we have assumed that the materials are linear, homogeneous and isotropic, which are possibly unrealistic for these types of rock.

For waves in elastic materials, two different wave speeds are of importance.
The pressure-, $c_p$, and shear-wave speed, $c_s$.
These relate to the material parameters as
\begin{equation}
c_p=\sqrt{\frac{\lambda+2\mu}{\rho}}, \qquad c_s=\sqrt{\frac{\mu}{\rho}}.
\end{equation}
The parameters in \eqref{eq:material parameters} correspond to the following wave-speeds
\begin{align}
&c_p=c_{p,1}=1.7728,    \quad &&c_{p,2}=2.3611 \nonumber\\
&c_s=c_{s,1}=1,    \quad &&c_{s,2}=1.2704. \nonumber
\end{align}

For time discretization we have used the explicit fourth order accurate classical Runge-Kutta.
In the experiment below a time step,
\begin{equation*}
\tau =0.2\frac{h}{p^2} \left(\max_{\Omega}(c_p)\right)^{-1},
\end{equation*}
has been used.
Since the condition number of the mass matrix is expected to be large a direct solver was used to invert $\mathcal{M}$ during the time stepping of \eqref{eq:discrete_system}.
\subsection{Convergence for the Single Domain Problem\label{sec:experiment_single_convergence}}
Assume that we have an elastic pressure wave traveling through $\R^2$ in the $x$-direction:
\begin{equation}
u_1^{in}(x,t)=\cos(\omega(t-x/c_p)), \quad u_2^{in}=0.
\label{eq:plane_wave}
\end{equation}
Here, $\omega$ is a constant which we choose as $\omega=\pi$.
Let this wave hit a circular inclusion (vacuum inside) with radius, $R=1$.
At the boundary of the inclusion, $\Gamma_N$, a homogeneous Neumann boundary condition is enforced.
If we consider this problem in all of $\R^2$ the total solution, $\tilde{u}$, will be the sum of the incoming, $u^{in}$, and reflected wave, $u^{ref}$:
\begin{equation}
 \tilde{u}=u^{in}+u^{ref}.
\end{equation}
The reflected wave can be computed analytically. 
The total analytical solution (given in \cite{virta2015formulae}) is periodic in time and can be written as a series expansion in Bessel and Hankel functions. 
In this paper, we truncate the series and use it as our solution $\tilde{u}$.
Since the solution is rather complicated we do not restate the series-expansion here, but merely refer the interested reader to \cite{virta2015formulae}.

Consider now the single domain problem in \eqref{eq:EWESingleDomain}--\eqref{eq:EWE_Dir} posed on the finite domain as in Figure~\ref{fig:single_domain}.
We have a finite square domain with side length $L=2\pi$.
As in Figure~\ref{fig:aligned_immersed_parts}, the outer boundary is aligned with the mesh but the inner boundary is not.
We want to make the solution, $u$, on this truncated domain equal to the analytical solution, $\tilde{u}$, on $\R^2$.
To achieve this, we set the initial conditions equal to $\tilde{u}$:
\begin{equation}
\evaluated{u}{t=0}=\evaluated{\tilde{u}}{t=0}, \quad
\evaluated{\pdd{u}{t}}{t=0}=\evaluated{\pdd{\tilde{u}}{t}}{t=0},
\label{eq:initial_equal_analytical}
\end{equation}
and impose a Dirichlet boundary condition on the outer boundary equal to $\tilde{u}$:
\begin{equation}
\quad\evaluated{u}{\Gamma_D}=\tilde{u}.
\label{eq:dirichlet_equal_analytical}
\end{equation}
We solve this problem until the end time $T=2$ (corresponding to one period) and compute the $L_2$-error for decreasing mesh sizes.
Snapshots of the solution at the initial time and a quarter of a period later are shown in Figure~\ref{fig:inclusion_snapshots}.

The error in $L_2$-norm as a function of element size is shown in Figure~\ref{fig:convergence_inclusion} for $Q_1$- to $Q_3$-elements.
The straight lines in the figure denote the expected order of accuracy.
We see that the order is a bit low for large $h$, but when going to finer $h$ we get the expected order or even slightly higher order than expected.

\begin{figure}
\centering
  \includegraphics[width=.45\paperwidth]{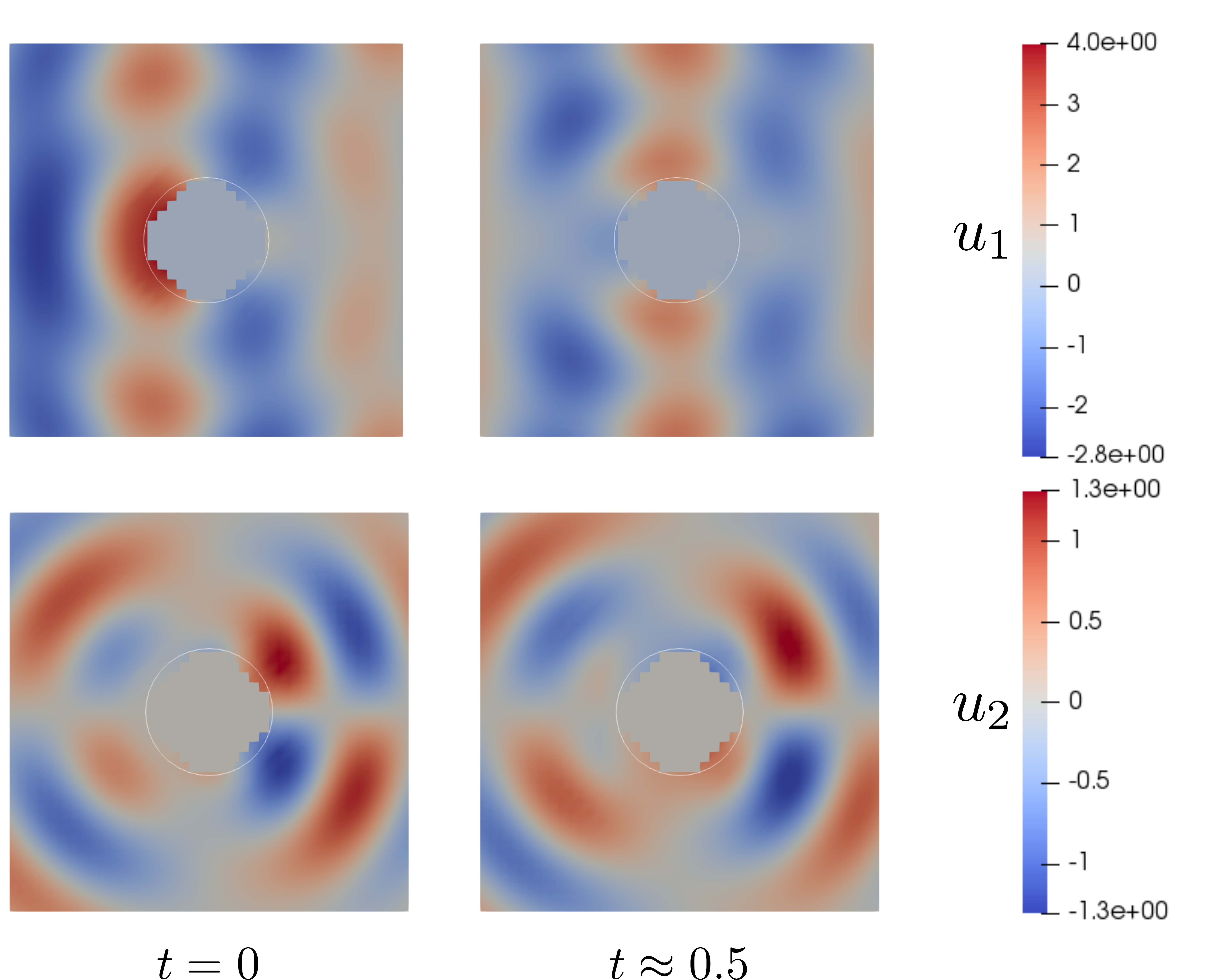}
  \caption{Snapshots of the solved single domain problem \label{fig:inclusion_snapshots}}
\end{figure}

\begin{figure}
\centerline{\includegraphics[width=.4\paperwidth]{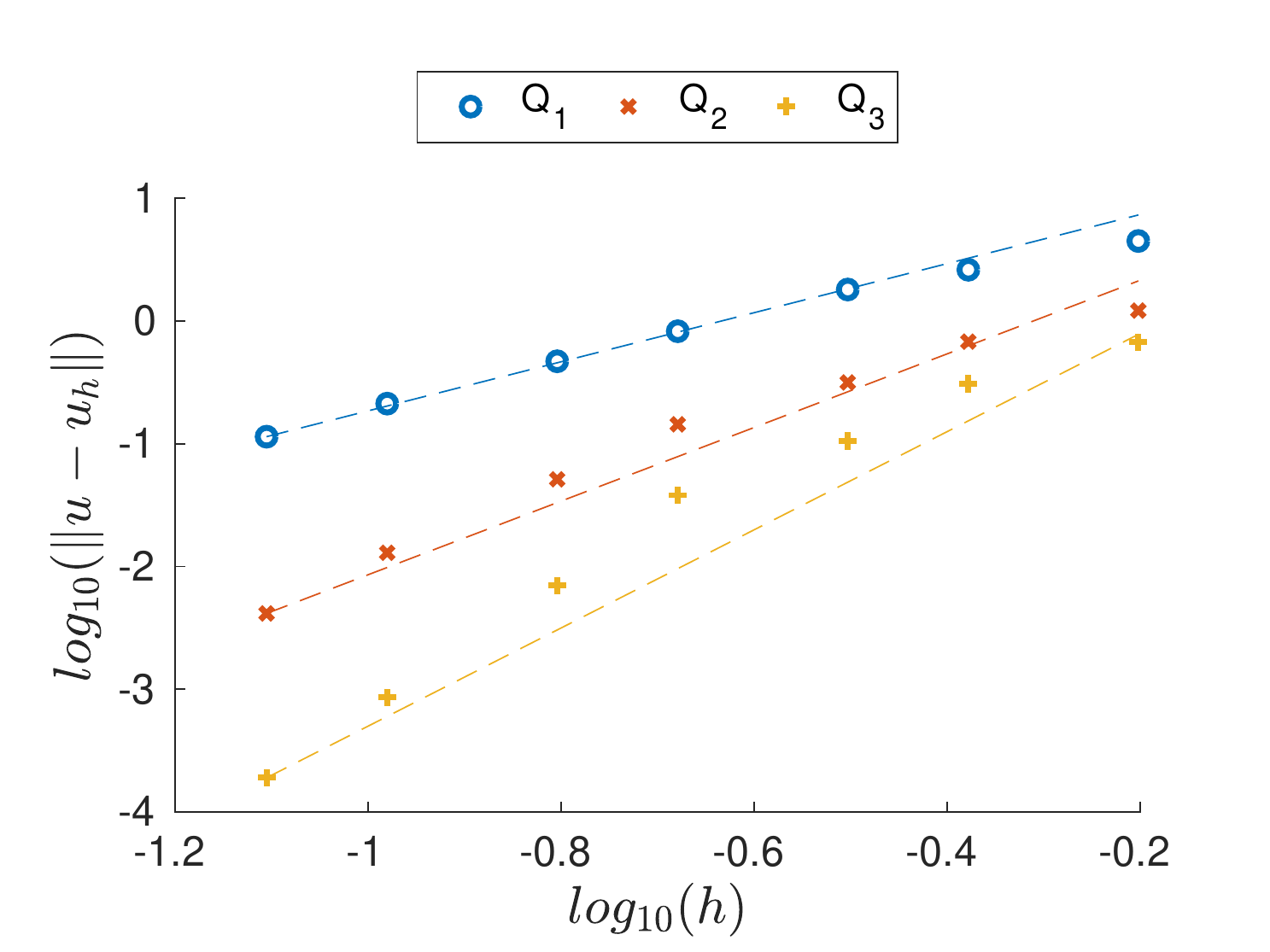}}
  \caption{$L_2$-error versus element size for the single domain problem, together with straight lines corresponding to the expected order of accuracy. \label{fig:convergence_inclusion}}
\end{figure}

\subsection{Convergence for the Interface Problem \label{sec:experiment_interface_convergence}}
Consider now a similar setup as in Section~\ref{sec:experiment_single_convergence}.
We have a plane wave of the form \eqref{eq:plane_wave} traveling through a material in $\R^2$ towards a disc.
The material has properties $\rho_1$, $\lambda_1$, $\mu_1$, and the disc has radius equal to 1.
However, instead of vacuum, we replace the material of the disc by another material with properties $\rho_2$, $\lambda_2$, $\mu_2$.
In the same way as before, the reflected wave can be solved for analytically and the total solution, $\tilde{u}$, can be found in \cite{virta2015formulae} in the form of a series expansion.
We again truncate the series and use it as our solution.

Now we solve the interface problem \eqref{eq:EWECompositeDomain}--\eqref{eq:EWE_Interface_Composite_2} posed on the finite domain in Figure~\ref{fig:interface_domain}.
Again we have a square domain with side length $2\pi$.
To make the solution of the problem equal to the analytical solution we again set the initial condition and the outer Dirichlet boundary condition equal to $\tilde{u}$, as in \eqref{eq:initial_equal_analytical}--\eqref{eq:dirichlet_equal_analytical}.
Snapshots of the solution at two different times are seen in Figure~\ref{fig:interface_snapshots}.
We see that the displacement in the $x$-direction looks like the plane wave in \eqref{eq:plane_wave}, but since the wave-speed is lower in $\Omega_2$ the plane wave gets distorted.

To verify the convergence we solve until the end time $T=2$ (corresponding to one period) and then compute the error. 
The error in $L_2$-norm as a function of element size is seen in Figure~\ref{fig:convergence_interface}.
We see that the order of accuracy is as expected for $Q_1$- and $Q_2$-elements.
For $Q_3$-elements the order is a bit low for large $h$, but eventually reaches the expected order when we go to finer $h$.

\begin{figure}
\centering
  \includegraphics[width=.45\paperwidth]{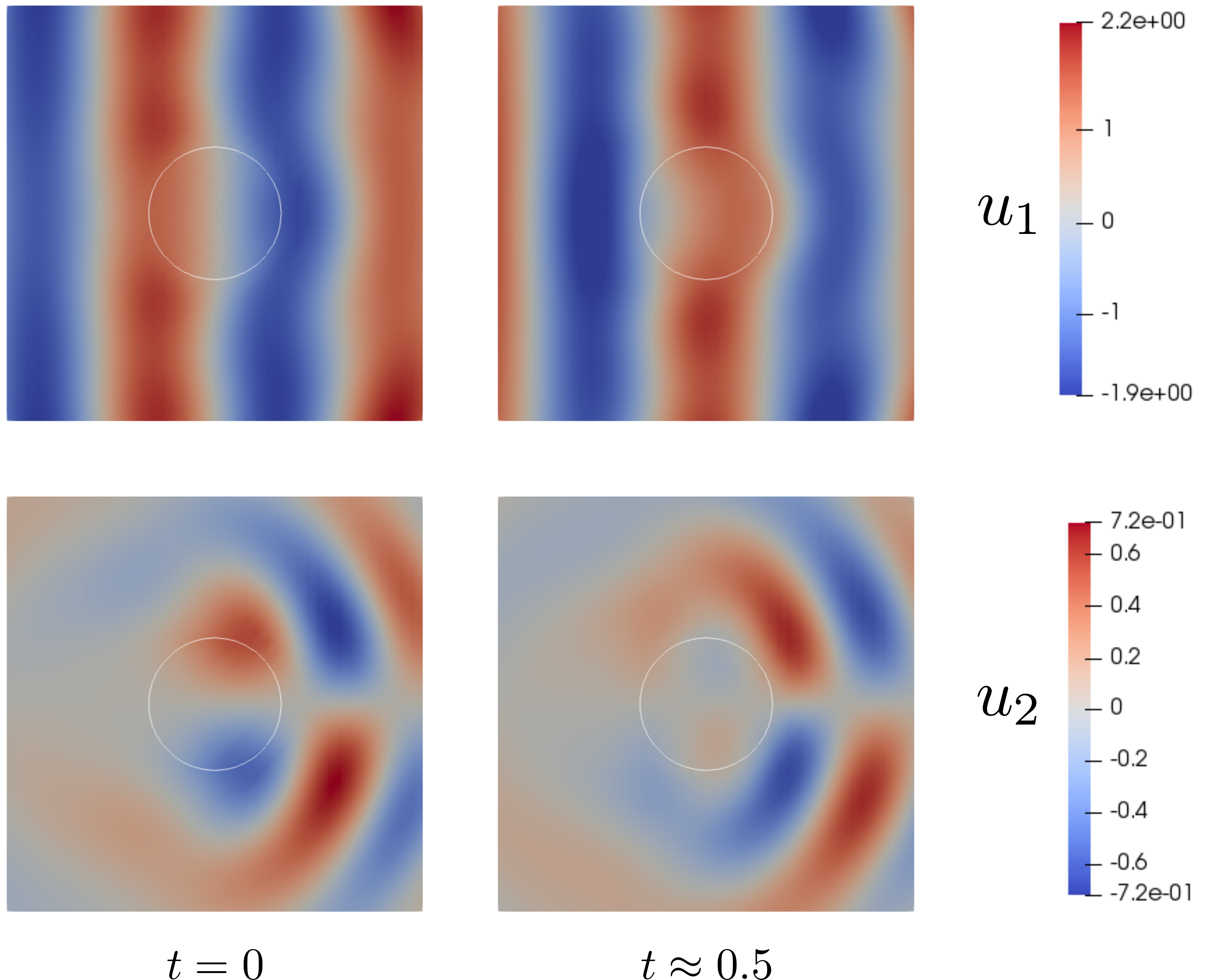}
  \caption{Snapshots of the solved interface problem \label{fig:interface_snapshots}}
\end{figure}

\begin{figure}
\centering
  \includegraphics[width=.4\paperwidth]{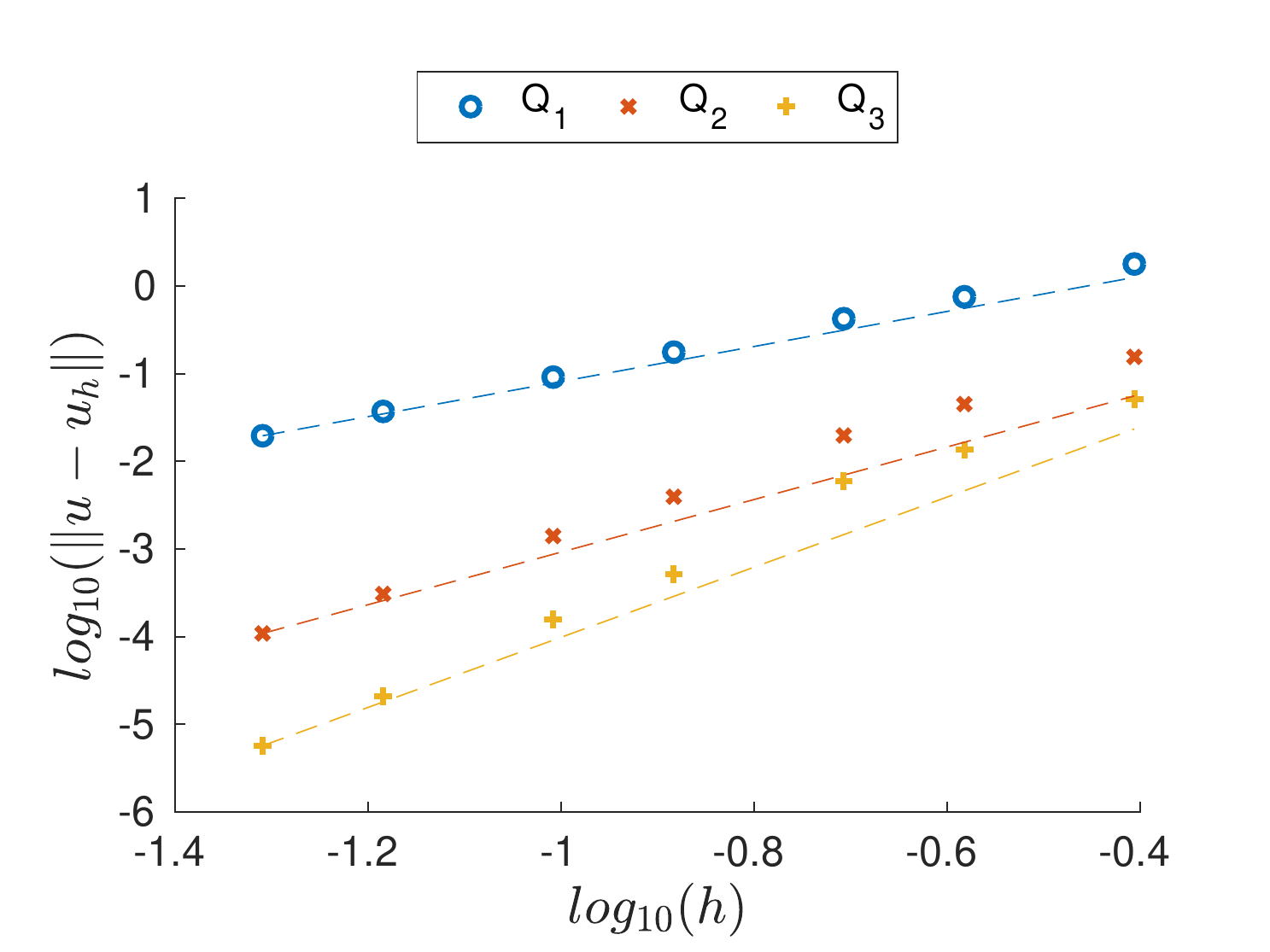}
  \caption{$L_2$-error versus element size for the interface problem,  together with straight lines corresponding to the expected order of accuracy. \label{fig:convergence_interface}}
\end{figure}

\subsection{Matrix Properties with Decreasing Cut-Size \label{sec:experiment_decreasing_cut}}
Consider the setup illustrated in Figure~\ref{fig:decreasing_cut_single} for the single domain and in Figure~\ref{fig:decreasing_cut_interface} for the interface problem.
For both setups, we have a rectangular domain on top of a square grid.
For the single domain problem in Figure~\ref{fig:decreasing_cut_single} the left, bottom and top boundary are aligned with the mesh, but the right domain boundary intersects the last column of elements with a cut of size $h_{cut}$.
For the interface problem, all boundaries are aligned with the mesh boundaries, but the immersed interface intersects the middle column of elements with a cut of size $h_{cut}$.
We are now interested in how the properties of the mass and stiffness matrix change when we vary the size of $h_{cut}$.
In the experiment, we use a background mesh containing $9\times 9$ elements, which is slightly finer than what is illustrated in Figure~\ref{fig:decreasing_cut}.

\begin{figure}[H]
\begin{subfigure}[b]{.3\paperwidth}
  \centering
  \includegraphics[width=.17\paperwidth]{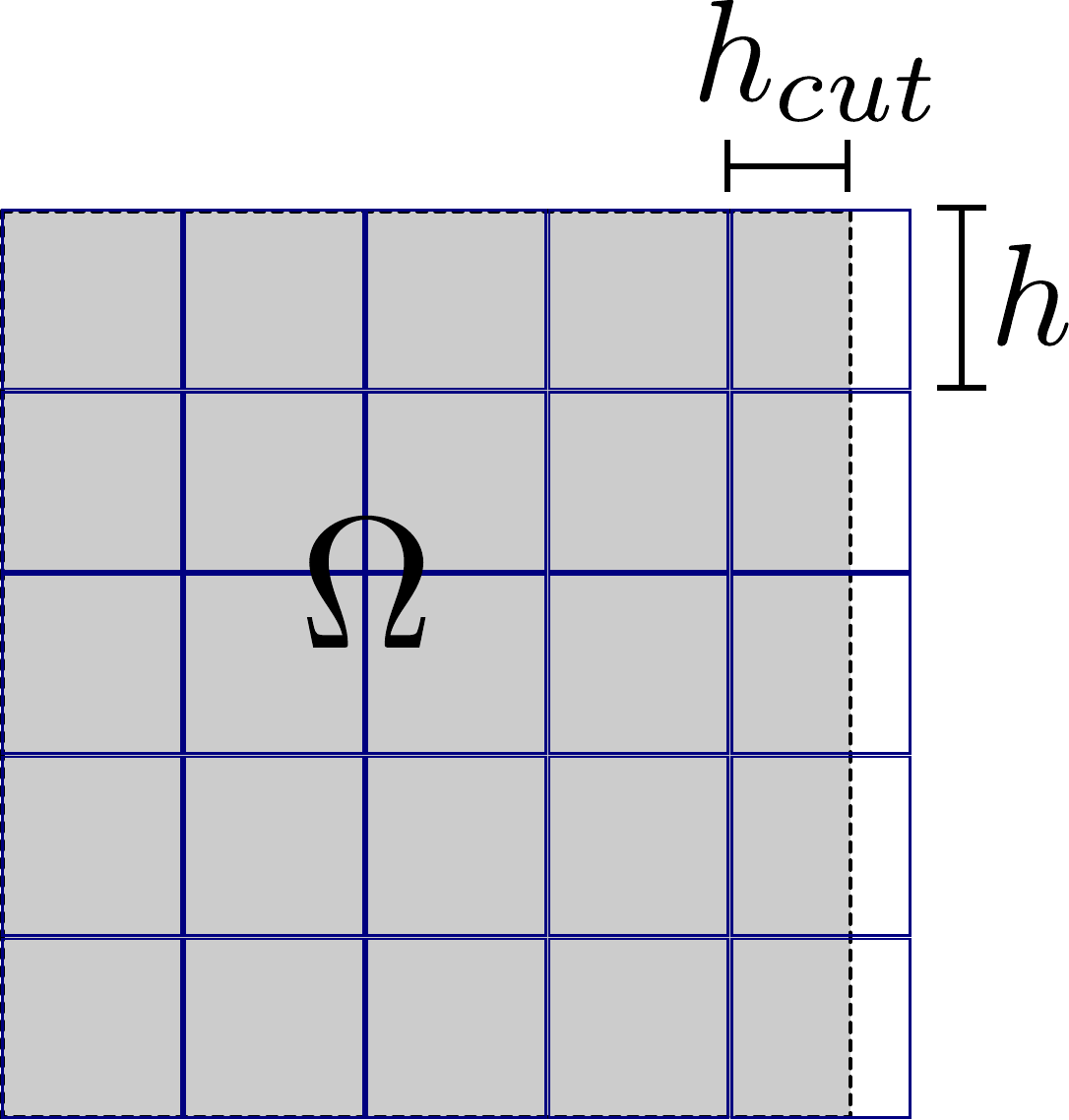}
  \caption{Single domain problem\label{fig:decreasing_cut_single}}
\end{subfigure}
\begin{subfigure}[b]{.3\paperwidth}
  \centering
  \includegraphics[width=.17\paperwidth]{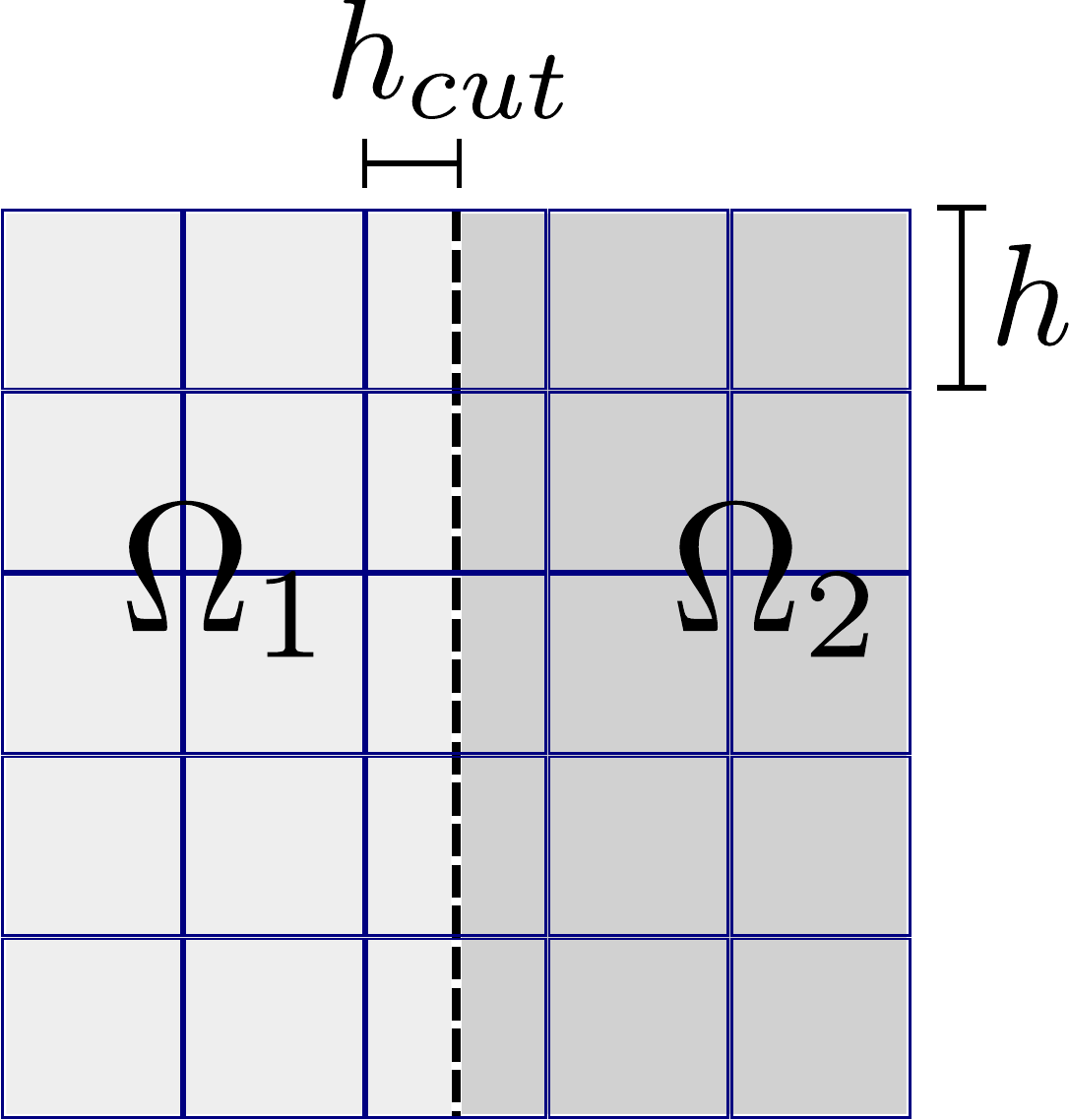}
  \caption{Interface problem\label{fig:decreasing_cut_interface}}
\end{subfigure}
\caption{Experiments where elements are intersected with a cut of size $h_{cut}$ \label{fig:decreasing_cut}}
\end{figure}

How the condition number of the mass matrix changes is seen in Figure~\ref{fig:cond_mass_decreasing_cut_single} for the single domain problem.
We see that when the cut size is large ($h_{cut}/h\approx 1$) the condition number is small and initially grows when $h_{cut}$ is decreased.
However, as the cut-size is decreased further the condition number becomes constant, as expected from the theory.
We also see that the constant level increases very fast when we increase the order of the elements, which is consistent with results previously presented in \cite{sticko2016higher,hansbo2017cut}.

In Figure~\ref{fig:cond_mass_decreasing_cut_interface} we see the condition number of the mass matrix for the interface problem. 
Note that we have $f(h_{cut}/h )$ on the $x$-axis, where
\begin{equation*}
f(x)=\log_{10}\left(x\right)-\log_{10}\left(1-x\right).
\end{equation*}
This makes the $x$-axis ``almost logarithmic'' as $h_{cut}/h$ approaches both $0$ and $1$,
since $f(x)$ is monotone on the interval $(0,1)$ and maps $(0,1)$ to $(-\infty,\infty)$.
In Figure~\ref{fig:cond_mass_decreasing_cut_interface} we see that the behavior is analogous to the single domain problem as $h_{cut}/h$ approaches 0.
We also see that the curve is almost mirrored in the point $h_{cut}=h/2$.
That the curve is not exactly mirrored can be explained by the difference in material parameters.  

In the same way, the condition number of the stiffness matrix is seen in Figure~\ref{fig:cond_stiffness_decreasing_cut_single} and \ref{fig:cond_stiffness_decreasing_cut_interface}.
We see that the dependence is similar as for the mass matrix in Figure~\ref{fig:cond_mass_decreasing_cut_single} and \ref{fig:cond_mass_decreasing_cut_interface}.

The $C_{FL}$-number computed from \eqref{eq:cfl_from_matrices} is shown in Figure~\ref{fig:cfl_decreasing_cut_single} for the single domain problem and in Figure~\ref{fig:cfl_decreasing_cut_interface} for the interface problem.
We see in the figures that the $C_{FL}$-number is completely independent of the size of the cut.
We also see that the $C_{FL}$-number becomes smaller when we increase the order of the elements.
This is also the case when using the standard (non-cut) finite element method.

\begin{figure}[H]
\begin{subfigure}[b]{.3\paperwidth}
  \centering
  \includegraphics[width=.3\paperwidth]{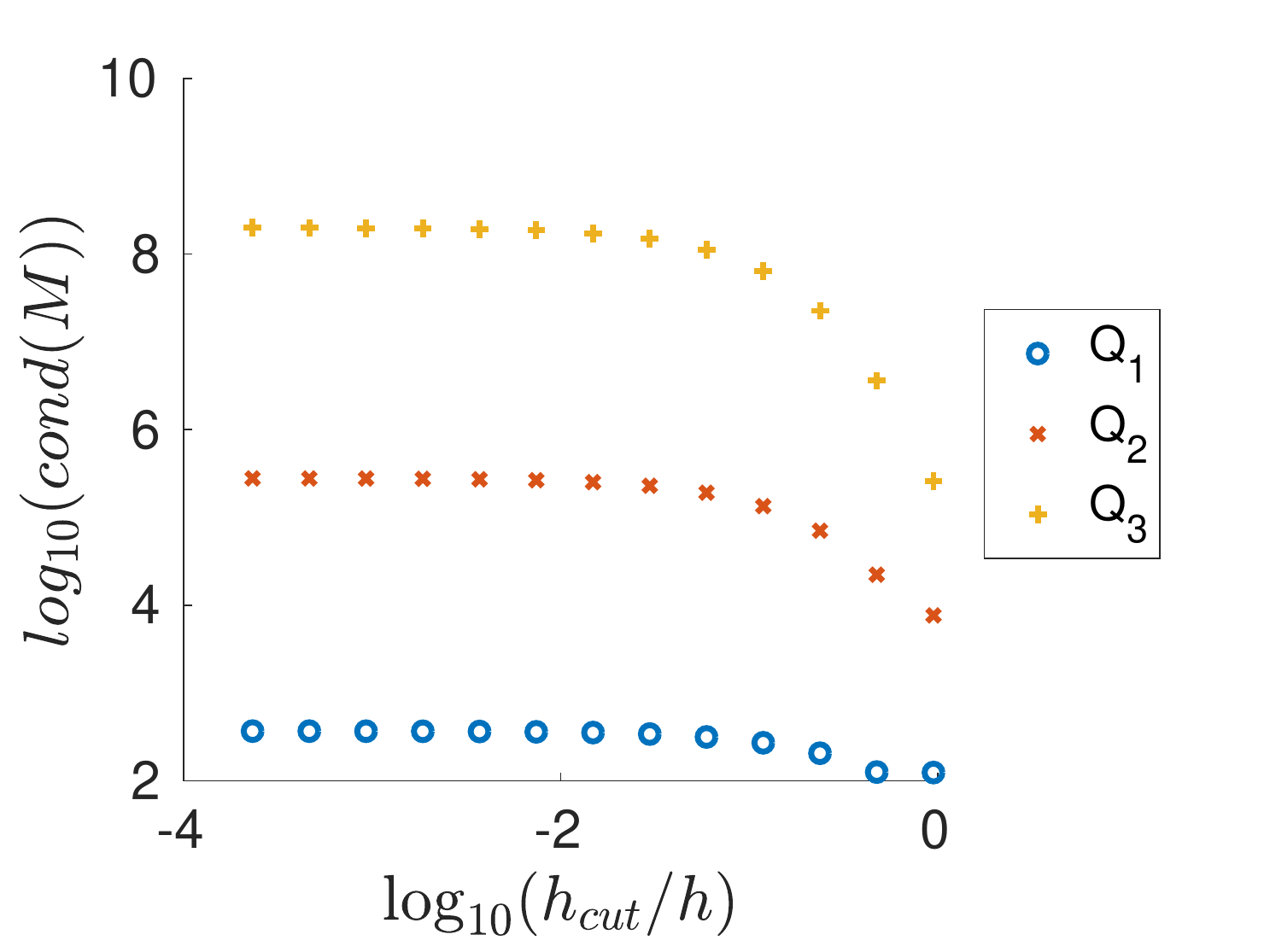}
  \caption{Single domain problem \label{fig:cond_mass_decreasing_cut_single}}
\end{subfigure}
\begin{subfigure}[b]{.3\paperwidth}
  \centering
  \includegraphics[width=.3\paperwidth]{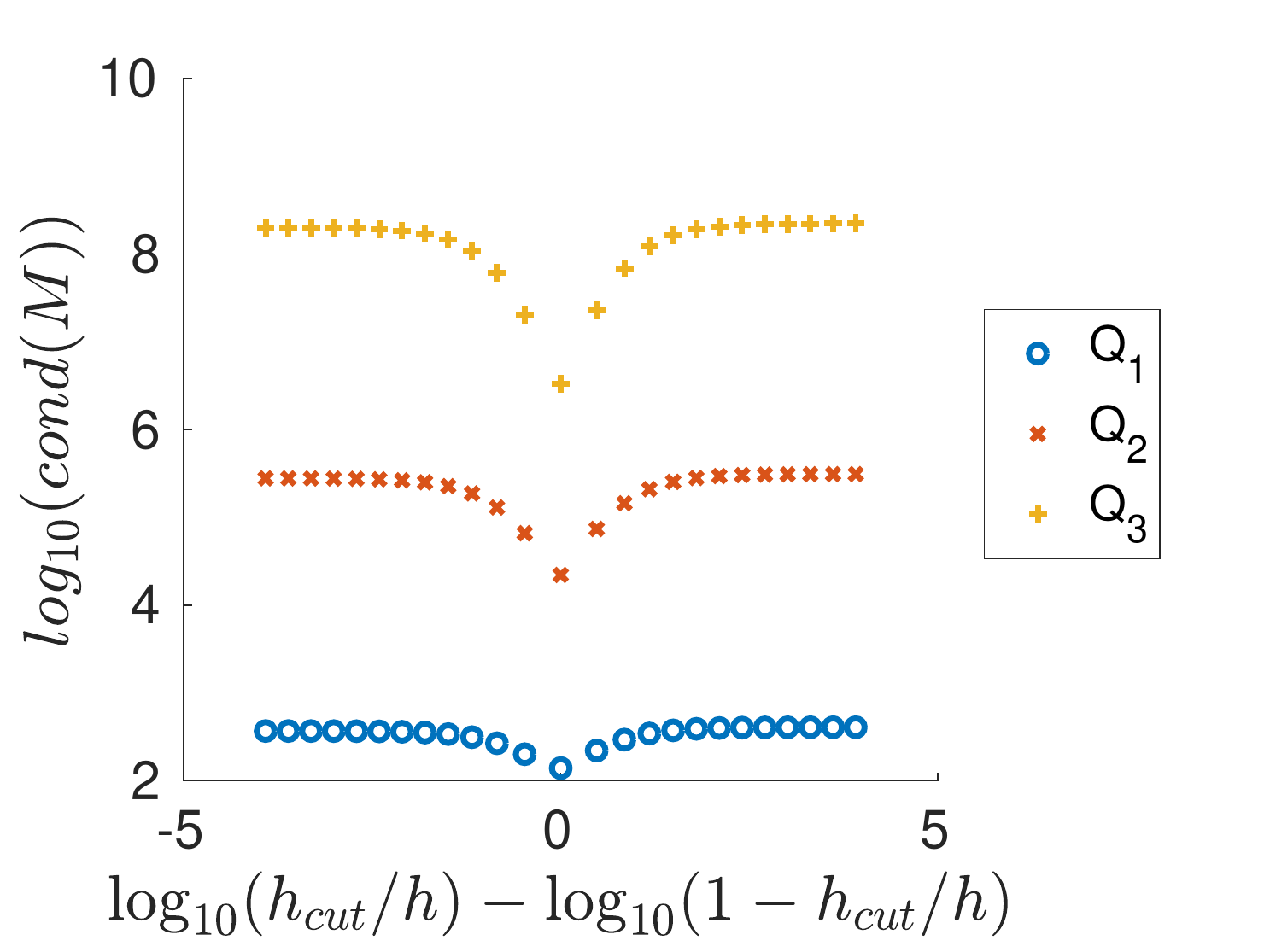}
  \caption{Interface problem \label{fig:cond_mass_decreasing_cut_interface}}
\end{subfigure}
  \caption{Condition number of the mass matrix when decreasing the size of $h_{cut}$ in Figure~\ref{fig:decreasing_cut} \label{fig:cond_mass_decreasing_cut}}
\end{figure}

\begin{figure}[H]
\begin{subfigure}[b]{.3\paperwidth}
  \centering
  \includegraphics[width=.3\paperwidth]{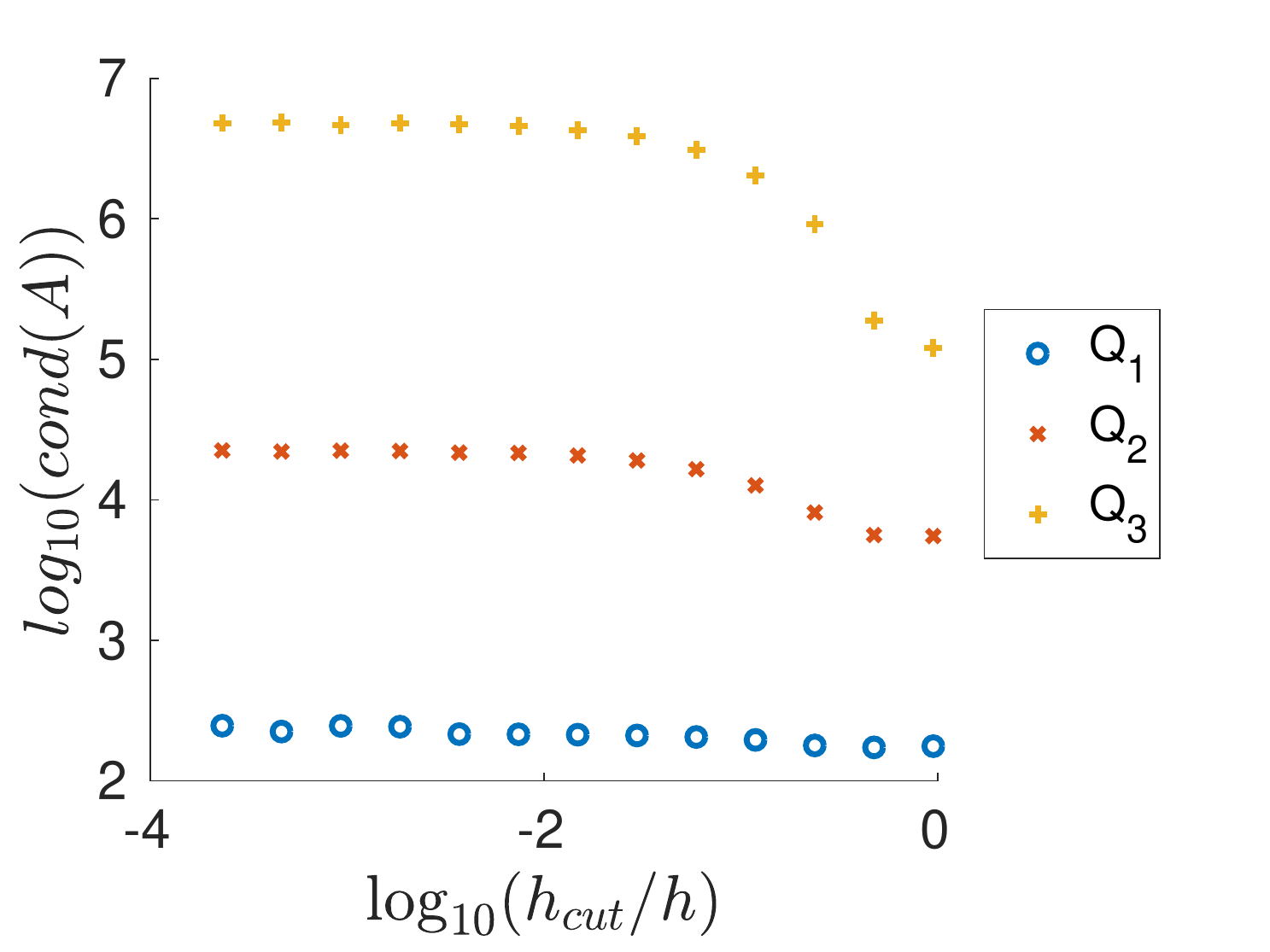}
  \caption{Single domain problem \label{fig:cond_stiffness_decreasing_cut_single}}
\end{subfigure}
\begin{subfigure}[b]{.3\paperwidth}
  \centering
  \includegraphics[width=.3\paperwidth]{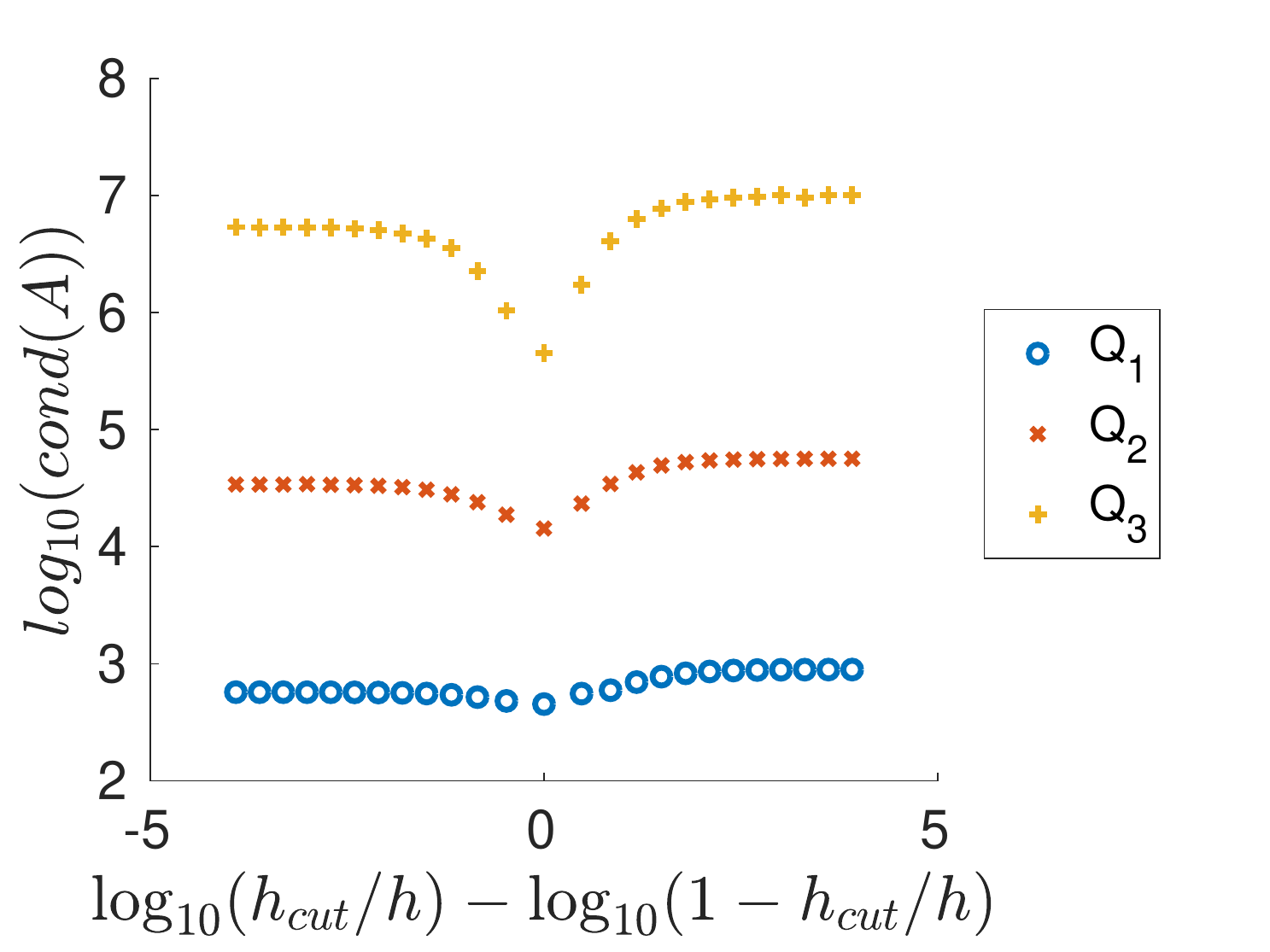}
  \caption{Interface problem\label{fig:cond_stiffness_decreasing_cut_interface}}
\end{subfigure}
  \caption{Condition number of the stiffness matrix when decreasing the size of $h_{cut}$ in Figure~\ref{fig:decreasing_cut} \label{fig:cond_stiffness_decreasing_cut}}
\end{figure}

\begin{figure}[H]
\begin{subfigure}[b]{.3\paperwidth}
  \centering
  \includegraphics[width=.3\paperwidth]{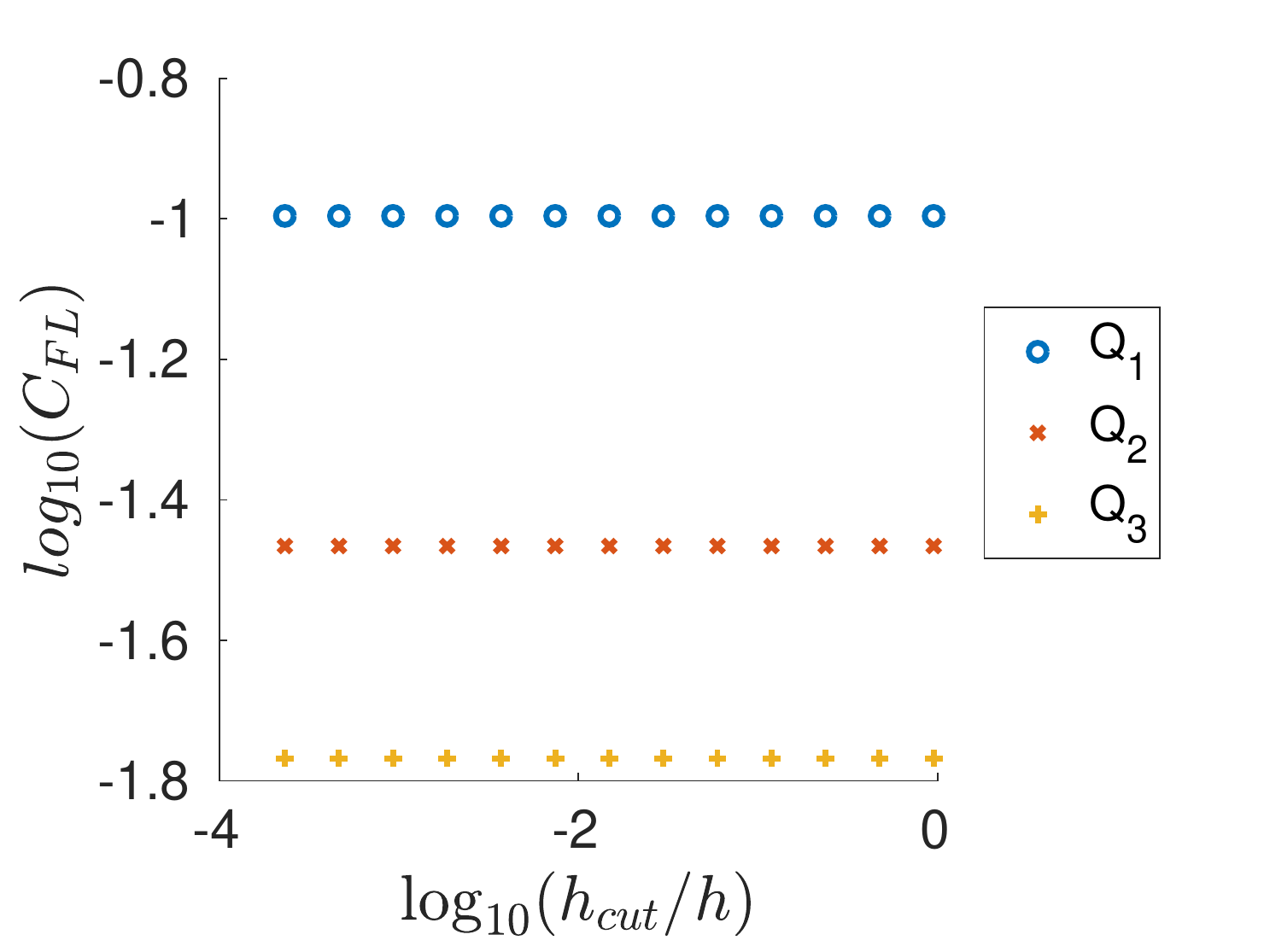}
  \caption{Single domain problem \label{fig:cfl_decreasing_cut_single}}
\end{subfigure}
\begin{subfigure}[b]{.3\paperwidth}
  \centering
  \includegraphics[width=.3\paperwidth]{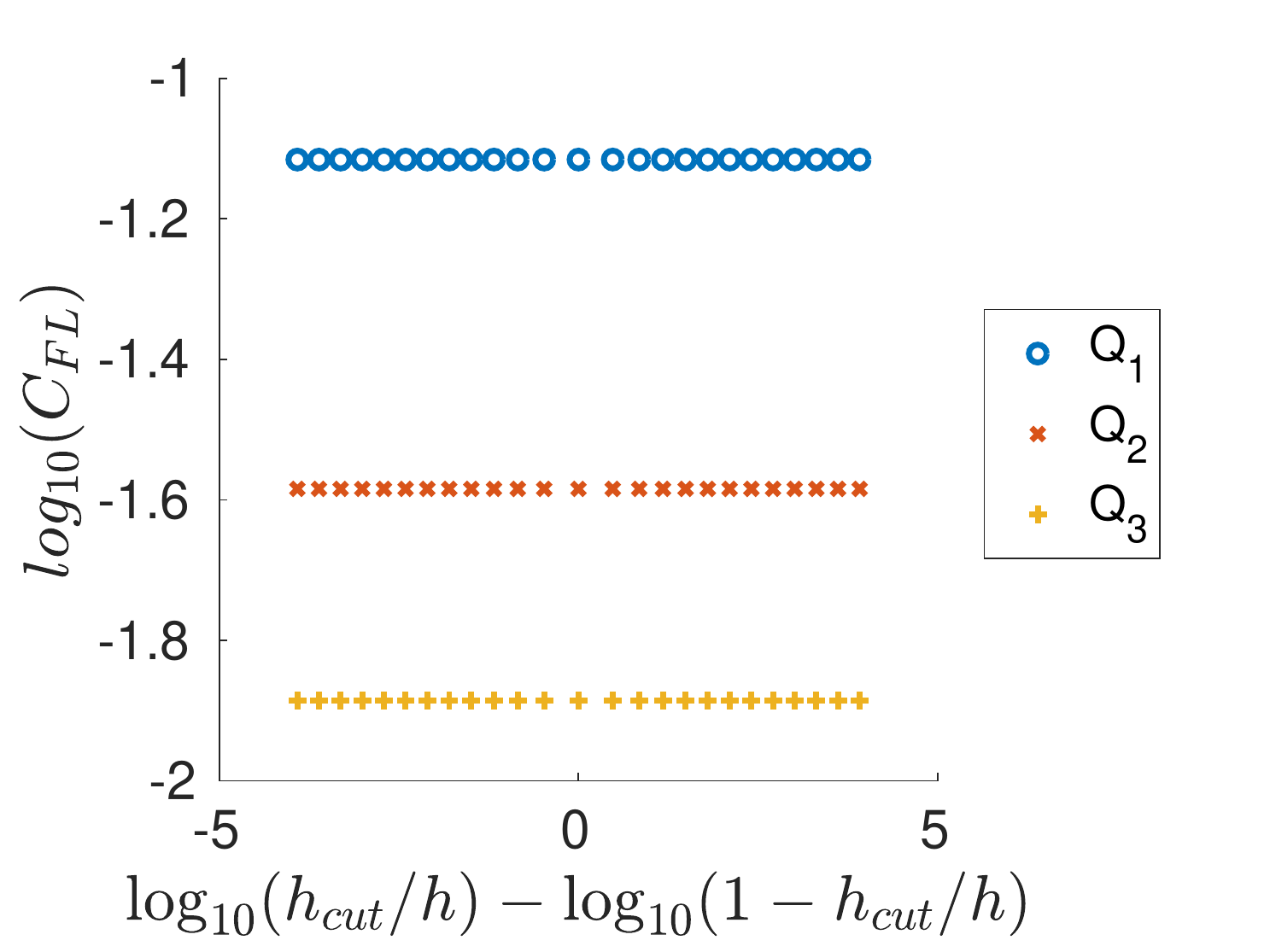}
  \caption{Interface problem \label{fig:cfl_decreasing_cut_interface}}
\end{subfigure}
  \caption{$C_{FL}$-number when decreasing the size of $h_{cut}$ in Figure~\ref{fig:decreasing_cut} \label{fig:cfl_decreasing_cut}}
\end{figure}

 \section{Discussion\label{sec:discussion}}
The numerical experiments in Section~\ref{sec:experiment_single_convergence} and \ref{sec:experiment_interface_convergence} show that the method converges with the orders expected from Theorem~\ref{thm:A_priori}.
Furthermore, from the experiment in Section~\ref{sec:experiment_decreasing_cut} we see that the method is robust when the size of the smallest cut in the mesh approaches zero.

The parameters \eqref{eq:material parameters} of the two materials used in the experiments for the interface problem are different but do not differ significantly.
A future possibility would be to test how more extreme differences in material parameters affect the performance of the method.
For the interface problem, the limit $\mu_2 \rightarrow 0$ is particularly important.
For this case, material 2 stops being elastic and the problem on $\Omega_2$ becomes equivalent to the acoustic wave equation \cite{monkola_numerical_2011}. 
One disadvantage of taking the limit $\mu_2 \rightarrow 0$ is that the problem on $\Omega_2$ still is a system.
Thus one future research direction would be to consider the problem of the elastic wave equation coupled directly with the acoustic wave equation.

The choice of numerical constants in front of $\gamma_M^i$, $\gamma_A^i$, $\gamma_I$ and $\gamma_D$ in \eqref{eq:stabilization_parameters}, \eqref{eq:kappas_convex_combination} and \eqref{eq:Nitsche_parameter} is rather arbitrary.
As far as we have seen the method is not particularly sensitive to the choice of constants.
Still one can wonder what happens when they are chosen differently.
If $\gamma_D$ and $\gamma_I$ are chosen too small coercivity is lost and the method becomes unstable, due to eigenvalues of the stiffness matrix becoming negative.
This has nothing to do with the method being immersed.
The same thing occurs also when symmetric Nitsche techniques are used in non-cut methods.
Generally one wants to choose $\gamma_D$ and $\gamma_I$ close to the stability limit.
If they are chosen larger than necessary the $C_{FL}$-number becomes smaller.
The influence of the stabilization parameters $\gamma_A^i$ and $\gamma_M^i$ on the condition numbers of the mass and stiffness matrix were discussed in \cite{Burman2012,sticko2016}, for linear $P_1$-elements.
There one could see that the condition numbers had a minimum when either stabilization parameter increased from 0.
However, the condition number of either matrix increased rather slowly after passing the minimum.
Thus, choosing $\gamma_M^i$ or $\gamma_A^i$ slightly larger than necessary does not have a severe effect.

As mentioned earlier, high order methods are typically attributed to being more efficient for hyperbolic problems.
We have not investigated whether this is the case for the present method, but there are several aspects that would affect the efficiency.
When increasing the order of elements the order of the quadrature must also be increased.
Creating quadrature rules on the intersected elements is typically expensive, and using more quadrature points means more work.
Whether it pays off to increase the order likely depends on what algorithm is being used to generate the quadrature. 
However, when solving wave propagation problems we are often interested in solving for an extended period of time.
When this is the case the time spent on time integration is typically dominant.
When time stepping \eqref{eq:discrete_system} we need to be able to invert the mass matrix.
If the number of degrees of freedoms is not too large we can afford to factorize it.
Once factorized, inverting the mass matrix is very fast.
However, if the number of degrees of freedom is very large we are forced to use an iterative method.
This is potentially not efficient since we saw in Section~\ref{sec:experiment_decreasing_cut} that the condition number is very large when the element order is high.

\bibliographystyle{spmpsci}
\bibliography{references}

\end{document}